\documentclass[a4paper,10pt]{amsart}

\DeclareRobustCommand{\SkipTocEntry}[5]{}

\usepackage{upgreek,mathtools,stmaryrd,enumitem,bbm} 
\usepackage[urlcolor=blue]{hyperref}
\usepackage{float}
\usepackage[numbers]{natbib} 

\usepackage{arydshln} 
\usepackage[textsize=footnotesize,color=yellow!70, bordercolor=white]{todonotes}

\usepackage{xcolor}
\definecolor{dblue}{rgb}{0,0,0.70}
\hypersetup{
	unicode=true,
	colorlinks=true,
	citecolor=dblue,
	linkcolor=dblue,
	anchorcolor=dblue
}

\usepackage
[a4paper, margin=1.2in]{geometry}

\usepackage[latin1]{inputenc}
\usepackage[T1]{fontenc}
\usepackage[english]{babel}

\usepackage{mathrsfs}
\usepackage{amscd}
\usepackage{amsfonts}
\usepackage{amsmath}
\usepackage{amssymb}
\usepackage{amstext}
\usepackage{amsthm}
\usepackage{amsbsy}

\usepackage{xspace}
\usepackage[all]{xy}
\usepackage{graphicx}
\usepackage{url}
\usepackage{latexsym}

\makeatletter
\newcommand*{\rom}[1]{\expandafter\@slowromancap\romannumeral {\sharp}1@}
\makeatother

\makeatletter

\renewcommand{\tocsection}[3]{%
  \indentlabel{\@ifnotempty{#2}{\bfseries\ignorespaces#1 #2\quad}}\bfseries#3}
\renewcommand{\tocsubsection}[3]{%
  \indentlabel{\@ifnotempty{#2}{\ignorespaces#1 #2\quad}}#3}
\renewcommand{\tocsubsubsection}[3]{%
  \indentlabel{\hspace{30pt}\@ifnotempty{#2}{\ignorespaces#1 #2\quad}}#3}

\newcommand\@dotsep{4.5}
\def\@tocline#1#2#3#4#5#6#7{\relax
  \ifnum #1>\c@tocdepth 
  \else
    \par \addpenalty\@secpenalty\addvspace{#2}%
    \begingroup \hyphenpenalty\@M
    \@ifempty{#4}{%
      \@tempdima\csname r@tocindent\number#1\endcsname\relax
    }{%
      \@tempdima#4\relax
    }%
    \parindent\z@ \leftskip#3\relax \advance\leftskip\@tempdima\relax
    \rightskip\@pnumwidth plus1em \parfillskip-\@pnumwidth
    #5\leavevmode\hskip-\@tempdima{#6}\nobreak
    \leaders\hbox{$\m@th\mkern \@dotsep mu\hbox{.}\mkern \@dotsep mu$}\hfill
    \nobreak
    \hbox to\@pnumwidth{\@tocpagenum{\ifnum#1=1\bfseries\fi#7}}\par
    \nobreak
    \endgroup
  \fi}
\AtBeginDocument{%
\expandafter\renewcommand\csname r@tocindent0\endcsname{0pt}
}
\def\l@subsection{\@tocline{2}{0pt}{2.5pc}{5pc}{}}
\makeatother

\theoremstyle{definition}

\newcommand{\Ord}{\mathrm{Ord}}
\newcommand{\tc}{\mathrm{tc}}

\newcommand{\WO}{\mathrm{WO}}
\newcommand{\LO}{\mathrm{LO}}
\newcommand{\wo}{\mathrm{wo}}
\newcommand{\cof}{\mathrm{cof}}

\newcommand{\ck}{\mathrm{ck}}

\renewcommand{\Col}{\mathrm{Col}}

\catcode`\<=\active \def<{
\fontencoding{T1}\selectfont\symbol{60}\fontencoding{\encodingdefault}}
\catcode`\>=\active \def>{
\fontencoding{T1}\selectfont\symbol{62}\fontencoding{\encodingdefault}}

\newcommand{\tmdate}[1]{\today}

\newcommand{\ran}{\mathrm{ran}}

\newcommand{\otp}{\mathrm{otp}} 
\newcommand{\PP}{\mathbb{P}} 
\newcommand{\QQ}{\mathbb{Q}}

\newcommand{\LL}{\mathbb{L}}
\newcommand{\NN}{\mathbb{N}}

\newcommand{\ZFC}{\mathsf{ZFC}}
\newcommand{\KP}{\mathsf{KP}}

\newcommand{\PD}{{\sf PD}} 
\newcommand{\MM}{\mathbb{M}}
\newcommand{\inter}{{\rm int}}

\newcommand{\ittm}{\rm{ittm}}

\newcommand{\borelomega}{Borel${}^{(\omega)}$}
\newcommand{\borelalpha}{Borel${}^{(\alpha)}$}

\newcommand{\borelgamma}{Borel${}^{(\gamma)}$}
\newcommand{\borellessgamma}{Borel${}^{({<}\gamma)}$}
\newcommand{\borelinfty}{Borel${}^{(\infty)}$}
\newcommand{\borel}{Borel${}^{({<}\omega_1)}$}

\DeclareMathOperator{\rank}{\mathrm{rank}}

\newtheorem{fact}{Fact}[section]

\newtheorem{theorem}[fact]{Theorem}
\newtheorem{lemma}[fact]{Lemma}
\newtheorem{proposition}[fact]{Proposition}

\newtheorem{definition}[fact]{Definition}

\newtheorem{question}[fact]{Question}
\newtheorem{problem}[fact]{Problem}

\newtheorem*{problem*}{Problem}
\newtheorem*{generalproblem*}{General Problem}
\newtheorem*{problem A}{Problem 1}
\newtheorem*{problem B}{Problem 2}
\newtheorem*{claim*}{Claim}

\theoremstyle{remark} 
\newtheorem{remark}[fact]{Remark}

\newenvironment{enumerate-(a)}{\begin{enumerate}[label={\upshape (\alph*)}, leftmargin=2pc]}{\end{enumerate}}
\newenvironment{enumerate-(a)-r}{\begin{enumerate}[label={\upshape (\alph*)}, leftmargin=2pc,resume]}{\end{enumerate}}
\newenvironment{enumerate-(A)}{\begin{enumerate}[label={\upshape (\Alph*)}, leftmargin=2pc]}{\end{enumerate}}
\newenvironment{enumerate-(A)-r}{\begin{enumerate}[label={\upshape (\Alph*)}, leftmargin=2pc,resume]}{\end{enumerate}}
\newenvironment{enumerate-(i)}{\begin{enumerate}[label={\upshape (\roman*)}, leftmargin=2pc]}{\end{enumerate}}
\newenvironment{enumerate-(i)-r}{\begin{enumerate}[label={\upshape (\roman*)}, leftmargin=2pc,resume]}{\end{enumerate}}
\newenvironment{enumerate-(I)}{\begin{enumerate}[label={\upshape (\Roman*)}, leftmargin=2pc]}{\end{enumerate}}
\newenvironment{enumerate-(I)-r}{\begin{enumerate}[label={\upshape (\Roman*)}, leftmargin=2pc,resume]}{\end{enumerate}}
\newenvironment{enumerate-(1)}{\begin{enumerate}[label={\upshape (\arabic*)}, leftmargin=2pc]}{\end{enumerate}}
\newenvironment{enumerate-(1)-r}{\begin{enumerate}[label={\upshape (\arabic*)}, leftmargin=2pc,resume]}{\end{enumerate}}
\newenvironment{itemizenew}{\begin{itemize}[leftmargin=2pc]}{\end{itemize}}

\begin{document}




\author{Merlin Carl}
\address{}
\email{}
\urladdr{}

\author{Philipp Schlicht}
\address{}
\email{}
\urladdr{}

\author{Philip Welch}
\address{}
\email{}
\urladdr{}

\thanks{This project has received funding from the European Union's Horizon 2020 research and innovation programme under the Marie Sk{\l}odowska-Curie grant agreement No 794020 of the second-listed author (Project \emph{IMIC: Inner models and infinite computations}). The second-listed author was partially supported by FWF grant number I4039. 
This research was funded in whole or in part by EPSRC grant number EP/V009001/1 of the second- and third-listed authors. 
For the purpose of open access, the authors have applied a 'Creative Commons Attribution' (CC BY) public copyright licence to any Author Accepted Manuscript (AAM) version arising from this submission.
}

\title{Countable ranks at the first and second projective levels} 
\date{\today}

\begin{abstract} 
A rank is a notion in descriptive set theory that describes ranks such as the   Cantor-Bendixson rank on the set of closed subsets of a Polish space, differentiability ranks on the set of differentiable functions in $C[0,1]$ such as the Kechris-Woodin rank and many other ranks in descriptive set theory and real analysis. 
The complexity of many natural ranks is $\Pi^1_1$ or $\Sigma^1_2$. 
We propose to understand the least length of ranks on a set as a measure of its complexity. 
Therefore, 
the aim is to understand which lengths such ranks may have. 
The main result determines the suprema of lengths of countable ranks at the first and second projective levels. 
Furthermore, we characterise 
the existence of countable ranks on specific classes of $\Sigma^1_2$ sets. 
The connections arising between $\Sigma^1_2$ sets with countable ranks on the one hand and $\Sigma^1_2$ Borel sets on the other  
lead to 
a conjecture that 
unifies several results in descriptive set theory such as the Mansfield-Solovay theorem and a recent result of Kanovei and Lyubetsky. 
\end{abstract} 

\maketitle

\thispagestyle{plain} 


\setcounter{tocdepth}{1}
\tableofcontents

\section{Introduction}

\subsection{Background}

A \emph{rank} is a layering of a set of reals that can arise, for instance, from a transfinite iteration of a derivation process. 
It represents a set of reals as a union of a chain of subsets such that each layer is less complex than the whole set. 
A precise definition is given in Section \ref{section ranks}. 
A basic example is the Cantor-Bendixson rank of a closed set, which is obtained by iteratively removing isolated points. 
Ranks are ubiquitous in real analysis, for example ranks on the set of closed sets of uniqueness for trigonometric series following work of Piatetski-Shapiro \cite{kechris1987descriptive}, 
Kechris-Woodin ranks on sets of differentiable functions \cite{westrick2014lightface} and 
oscillation ranks on sets of pointwise convergent sequences of continuous functions \cite{kechris2012classical}. Examples in topological dynamics include ranks on distal flows 
using Furstenberg's structure theorem \cite{beleznay1995collection} and ranks on shifts of finite type \cite{westrick2019topological}. 
These are all $\Pi^1_1$-ranks, so in particular their layers are Borel. 
Ranks are of interest at higher complexities as well. 
Martin and Moschovakis constructed ranks on projective sets assuming the axiom of projective determinacy \cite{Mo09}. 
The use of ranks to build trees \cite{St08} is relevant for the core model induction 
in inner model theory \cite[Section 3.3]{SchSt}. 
A further important application of ranks appears in Hjorth's analysis of the strength of ${\bf \Pi}^1_2$ Wadge determinacy \cite{Hj96}.  

We first describe two examples of ranks on $\Pi^1_1$ sets from descriptive set theory. 
The set $\WO$ of all wellorders on $\NN$ supports a natural rank with layers $\WO_\alpha$ of all wellorders with order type $\alpha$. 
The Cantor-Bendixson rank on the $\Pi^1_1$ set of countable closed subsets of a Polish space is defined by forming a sequence of derivatives. 
In successor steps, the isolated points of a countable closed set $C$ are removed to obtain its derivative $C'$. 
In limit steps, one forms the intersection. 
This defines a sequence of iterated derivatives $C^{(\alpha)}$ of $C$ and its rank is by definition the least $\alpha$ with $C^{(\alpha)}=\emptyset$, 
so the $\alpha$-th layer consists of all countable closed sets of rank $\alpha$. 



In real analysis, the Kechris-Woodin rank is defined on the $\Pi^1_1$ set of differentiable functions in $C[0,1]$ \cite[Section 34.F]{kechris2012classical}. 
Let $\Delta_f(x,y)= \frac{f(x)-f(y)}{x-y}$ denote the slope between $x,y \in [0,1]$. 
Kechris and Woodin define a \emph{derivative} as follows. 
One removes points in a closed set $C$ at which $f$ is close to being differentiable in the sense that the oscillation of $f'$ at $x$ is no more than $\epsilon$. 
The derivative $C'_{f,\epsilon}$ is defined as the remainder consisting of all $x\in C$ such that for all $\delta>0$, there exist rational intervals $[p,q]$ and $[r,s]$ in $B(x,\delta)\cap [0,1]$ such that 
$[p,q] \cap [r,s] \cap C\neq\emptyset$ and $|\Delta_f(p,q)- \Delta_f(r,s)|\geq \epsilon$. 
Starting from $[0,1]$, one defines a sequence  by $D^0_{f,\epsilon}=[0,1]$, $D^{\alpha+1}_{f,\epsilon}=(D^\alpha_{f,\epsilon})'_{f,\epsilon}$ and $D^\lambda_{f,\epsilon}= \bigcap_{\alpha<\lambda}D^{\alpha+1}_{f,\epsilon}$ for limits $\lambda$. 
The Kechris-Woodin rank $|f|$ of a differentiable function $f \in C[0,1]$ is the least ordinal $\alpha$ such that $D^\alpha_{f,\epsilon}=\emptyset$ for all $\epsilon>0$. 
For example, any continuously differentiable function $f$ has rank $1$.  

What can be said about the lengths of ranks from an abstract viewpoint? 
While any $\Pi^1_1$ or $\Sigma^1_2$ set admits some rank \cite{Mo09}, this does not say which sets admit ranks of which lengths. 
We here focus on ranks of length at most $\omega_1$, since most natural ranks have this property; the length of a rank given by either an infinite derivation or an infinite time computation \cite{decisiontimes} is typically at most $\omega_1$, since a terminating derivation or computation with real input is countable.\footnote{While this holds for many natural derivatives and notions of computation, it is not meant as a precise mathematical statement. 
It is easy to see for infinite computations as described in Section \ref{section - infinite time}.}  
Therefore, our first aim is to determine the supremum of those countable ordinals that can arise as the length of a rank. 
We are further interested in connections between the complexity of a set and the length of ranks that it supports. 
For instance, a $\Pi^1_1$ set admits a countable rank, i.e. one of countable length, if and only if it is Borel. 
Thus the least length of ranks on a $\Pi^1_1$ set can be understood as a measure of complexity 
that generates a hierarchy of $\Pi^1_1$ Borel sets. 
For $\Sigma^1_2$ sets, the above equivalence fails. 
We therefore aim to characterise those $\Sigma^1_2$ sets that support countable ranks. 
This finally leads to problems about $\Sigma^1_2$ Borel sets that are related to the study of $\Pi^1_1$ Borel sets initiated by Kechris, Marker and Sami \cite{MR1011178}. 

\subsection{Results}

In our main result, we determine the suprema of lengths of countable $\Pi^1_1$ ranks and of $\Sigma^1_2$ ranks. 
Since these are the same ordinals, the suprema are the same for all classes in between as well. 
To state the result, we write $\tau$ for the supremum of $\Sigma_2$-definable ordinals over $L_{\omega_1}$.\footnote{$\omega_1$ always denotes $\omega_1^V$.} 
We shall give a more useful definition of $\tau$ in Section \ref{subsection Kechris ordinal}. 
This is a variant of stable ordinals from proof theory\footnote{see \cite[Section 5]{rathjen2017higher}.} that we call \emph{robust}. 
We obtain a number of variants of this result. 
For instance, the suprema of ranks of $\Sigma^1_2$ wellfounded relations and lengths of $\Pi^1_1$ prewellorders\footnote{A \emph{prewellorder} is a wellfounded linear quasiorder.} on $\Pi^1_1$ sets, among others, also equal $\tau$. 

\begin{theorem} 
\label{main theorem} 
The following sets of ordinals all have (strict) supremum\footnote{I.e., the least strict upper bound.} $\tau$: 
\begin{enumerate-(1)} 
\item 
\label{main theorem 1} 
\begin{enumerate-(a)} 
\item 
\label{main theorem 1a} 
$\Pi_1$-definable ordinals over $L_{\omega_1}$ 
\item 
\label{main theorem 1b} 
$\Sigma_2$-definable ordinals over $L_{\omega_1}$ 
\item 
\label{main theorem 1c} 
Least elements of nonempty 
$\Pi_1$-definable subsets of $\omega_1$ over $H_{\omega_1}$ 
\item 
\label{main theorem 1d} 
As in \ref{main theorem 1c}, but for $\Sigma_2$-definable subsets 
\end{enumerate-(a)} 
\item 
\label{main theorem 3} 
Countable ranks of $\Sigma^1_2$ wellfounded relations 
\item 
\label{main theorem 2} 
Lengths of countable  
\begin{enumerate-(a)} 
\item 
\label{main theorem 2a} 
$\Pi^1_1$ ranks 
\item 
\label{main theorem 2c} 
$\Sigma^1_2$ ranks 
\end{enumerate-(a)} 
\item 
\label{main theorem 22}
Lengths of countable 
\begin{enumerate-(a)} 
\item 
\label{main theorem 22a} 
$\Sigma^1_1$ prewellorders on $\Sigma^1_1$ sets 
\item 
\label{main theorem 22b} 
(strict) $\Pi^1_1$ prewellorders on $\Pi^1_1$ sets\footnote{I.e., this holds for the supremum of lengths of prewellorders and similarly for strict prewellorders.}  
\item 
\label{main theorem 22c} 
strict $\Sigma^1_2$ prewellorders on $\Sigma^1_2$ sets 
\end{enumerate-(a)} 
\item 
\label{main theorem 6} 
$L$-levels of\footnote{Note that the $L$-levels of countable $\Pi^1_2$ sets are not strictly bounded by $\tau$ if $\omega_1^L=\omega_1$ by Proposition \ref{ctble Pi12 cofinal in tau}.} 
\begin{enumerate-(a)} 
\item 
countable 
$\Pi^1_1$ sets\footnote{I.e., ordinals $\alpha$ such that there exists a countable $\Pi^1_1$ set in $L_{\alpha+1}\setminus L_\alpha$.} 
\item 
countable $\Sigma^1_2$ sets
\item 
$\Pi^1_2$ singletons 
\end{enumerate-(a)} 
assuming $\omega_1^L=\omega_1$ 
\end{enumerate-(1)} 
\end{theorem} 

The equivalences in \ref{main theorem 1} are shown in Section \ref{subsection Kechris ordinal}, \ref{main theorem 3} in Section \ref{section upper bound}, \ref{main theorem 2} in Section \ref{section - the lower bound}, \ref{main theorem 2} and \ref{main theorem 22} in Section \ref{section - beyond tau} and \ref{main theorem 6} in Sections \ref{section robust vs stable} and \ref{subsection countable sets}. 
Note that  by \ref{main theorem 1}\ref{main theorem 1a}, the value of $\tau$ is variable: if $\omega_1^L=\omega_1$ then $\tau<\omega_1^L$, but if $\omega_1^L$ is countable then $\tau>\omega_1^L$. 
The next figure shows the size of $\tau$ in comparison with the relevant ordinals. 
Here $\sigma$ denotes the supremum of $\Sigma_1$-definable ordinals in $L_{\omega_1}$; it equals the supremum $\delta^1_2$ of order types of $\Delta^1_2$-definable wellorders on $\omega$. 

\newcommand{\cb}{}
\newcommand{\cre}{}

\bigskip 
\begin{figure}[H]
\begin{tikzpicture}[scale=0.8]
\small

\node [below] at (0,0) {{\cre $\omega_1 =\omega_1^L$}};

\draw (0,0) -- (0,4.3);

\draw (-0.1,0) -- (0.1,0); 

\draw (-0.1,0.5) -- (0.1,0.5); 
\node [right] at (0,0.5) {$\omega_1^{ck}$};

\draw (-0.1,1) -- (0.1,1); 
\node [right] at (0,1) {$\sigma=\delta^1_2$};

\draw (-0.1,1.5) -- (0.1,1.5); 
\node [right] at (0,1.5) {{\cre $\tau$}};

\draw (-0.1,2) -- (0.1,2); 
\node [right] at (0,2) {{\cre $\omega_1^L = \omega_1$}};

\draw (-0.1,3) -- (0.1,3); 
\node [right] at (0,3) {{\cb$\omega_2^L$}};

\node [below] at (2,0) {{\cre $\omega_1=\omega_2^L$}};

\draw (2,0) -- (2,4.3);

\draw (1.9,0) -- (2.1,0); 

\draw (1.9,0.5) -- (2.1,0.5); 
\node [right] at (2,0.5) {$\omega_1^{ck}$};

\draw (1.9,1) -- (2.1,1); 
\node [right] at (2,1) {$\sigma=\delta^1_2$};

\draw (1.9,2) -- (2.1,2); 
\node [right] at (2,2) {{\cb $\omega_1^L$}};

\draw (1.9,2.5) -- (2.1,2.5); 
\node [right] at (2,2.5) {{\cre $\tau$}};

\draw (1.9,3) -- (2.1,3); 
\node [right] at (2,3) {{\cre $\omega_2^L=\omega_1$}};

\node [below] at (4,0) {{\cre $\omega_1>\omega_2^L$}};

\draw (4,0) -- (4,4.3);

\draw (3.9,0) -- (4.1,0); 

\draw (3.9,0.5) -- (4.1,0.5); 
\node [right] at (4,0.5) {$\omega_1^{ck}$};

\draw (3.9,1) -- (4.1,1); 
\node [right] at (4,1) {$\sigma=\delta^1_2$};

\draw (3.9,2) -- (4.1,2); 
\node [right] at (4,2) {{\cb $\omega_1^L$}};

\draw (3.9,3) -- (4.1,3); 
\node [right] at (4,3) {{\cb$\omega_2^L$}};

\draw (3.9,3.5) -- (4.1,3.5); 
\node [right] at (4,3.5) {{\cre $\tau$}};

\draw (3.9,4) -- (4.1,4); 
\node [right] at (4,4) {\cre $\omega_1$};


\end{tikzpicture}
\caption{Comparison of $\tau$ with $\omega_1^L$ and $\omega_2^L$.}
\end{figure}
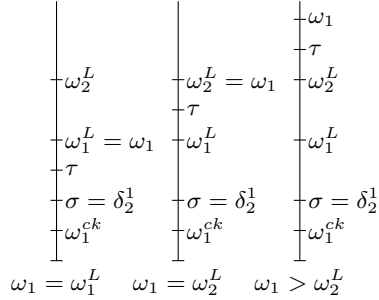 
However, the value of $\tau$ does not change in any generic extension $V[G]$ in comparison with the universe $V$ if $0^\#$ exists.
Then $\omega_1^{V[G]}$ is a Silver indiscernible and therefore $L_{\omega_1^V}$ is an elementary substructure of $ L_{\omega_1^{V[G]}}$.

Several results related to Theorem \ref{main theorem} were known. 
Kechris, Marker and Sami \cite{MR1011178} proved that the supremum in \ref{main theorem 1}\ref{main theorem 1c} equals those of Borel ranks of Borel $\Pi^1_1$ sets and of Borel ranks of equivalence classes of $\Sigma^1_1$ equivalence relations with all classes Borel of bounded rank.\footnote{I.e., the Borel ranks of the equivalence classes have a countable supremum.} 
The equivalence of  \ref{main theorem 1}\ref{main theorem 1a} and  \ref{main theorem 1}\ref{main theorem 1b} was shown in 
\cite[Lemma 4.1]{decisiontimes}. 
It was also known that the supremum of semidecision times of infinite time Turing machines equals the suprema in \ref{main theorem 1} \cite[Theorem 4.5]{decisiontimes}. 
The same result and proof work for ordinal time and tape Turing machines.


In the proofs, we obtain various new examples of sets that admit only ranks of certain minimal countable lengths, and even sets that admit only ranks of a unique countable length. 
Such sets are constructed at the levels of $\Pi^1_1$ and $\Sigma^1_2$. 
As a byproduct of the theorem, the equivalences in \ref{main theorem 1} show that the ordinal $\gamma^1_2$ studied by Kechris equals $\tau$.\footnote{See Section \ref{subsection Kechris ordinal} 
for a definition of $\gamma^1_2$.} 
\ref{main theorem 3} is an effective version of the Kunen-Martin theorem for $\Sigma^1_2$ relations of countable rank. 
Moreover, we immediately obtain the following known results as corollaries. 
Assuming $\omega_1$ is inaccessible in $L$, $\tau$ is a fixed point of the $\aleph$-function in $L$ \cite[Theorem 2.3 \& Corollary 2.4]{MR1011178} since for any ordinal $\alpha$ that is $\Sigma_2$-definable over $L_{\omega_1}$, $\aleph_\alpha$ is $\Sigma_2$-definable over $L_{\omega_1}$ as well. 
Clearly $\cof(\tau)^L=\omega$ \cite[Theorem 2.4]{MR1011178}. 

With the help of Theorem \ref{main theorem}, we shall determine the suprema of various classes of countable prewellorders in Section \ref{section - beyond tau}. 
Figure \ref{figure length of pwos} summarises what we know about these suprema. 
For the rightmost column, we assume that $0^\#$ exists and write $\iota_0$ for the first Silver indiscernible, while 
the remaining claims are proved in $\ZFC$. 
When writing ``(strict) prewellorders'', we mean that the supremum of the lengths of prewellorders and the one for strict ones have the same value. 
Similarly, when we write ``$\Delta^1_1$ pwo's on $2^\omega$/$\Delta^1_1$ sets'' in the topmost row, 
we mean the suprema of lengths of $\Delta^1_1$ pwo's on $2^\omega$ and those on arbitrary $\Delta^1_1$ sets both have supremum $\omega_1^\ck$. 
Note the periodic pattern in the columns of Figure \ref{figure length of pwos}. 

\begin{figure}[H] 
\begin{tabular}{ l  l  l  l } 
 & \hspace{48pt} $\omega_1^\ck$ & \hspace{52pt}  $\tau$ & \hspace{38pt}  ${>}\iota_0$ \\ 
 \hline 
$\Delta^1_1$ &  \emph{(strict)} pwo's on $2^\omega$/$\Delta^1_1$ sets  &   &   \\ 
\hdashline
$\Sigma^1_1$ & \emph{strict} pwo's on $2^\omega$/$\Sigma^1_1$ sets  &  pwo's on $2^\omega$/$\Sigma^1_1$ sets  &  \\ 
\hdashline
$\Pi^1_1$ & pwo's on $2^\omega$ &  pwo's on $\Pi^1_1$ sets   &   \\ 
 & & \emph{strict}  pwo's on $2^\omega$/$\Pi^1_1$ sets    &  \\ 
\hline
$\Delta^1_2$ &   & \emph{(strict)} pwo's on $2^\omega$/$\Delta^1_2$ sets   &   \\ 
\hdashline
$\Sigma^1_2$ & & \emph{strict} pwo's on $2^\omega$/$\Sigma^1_2$ sets  &  pwo's on $2^\omega$/$\Sigma^1_2$ sets \\ 
\hdashline
$\Pi^1_2$ & & pwo's on $2^\omega$ & pwo's on  $\Pi^1_2$ sets  \\ 
 & &  & \emph{strict}  pwo's on  $2^\omega$/$\Pi^1_2$ sets \\ 
\hline 
\end{tabular} 
\medskip 
\caption{Suprema of lengths of countable prewellorders.} 
\label{figure length of pwos} 
\end{figure} 

We investigate the lengths of ranks on countable and co-countable sets in Section \ref{subsection countable sets}. 
On the way, we study a problem that is of interest for its own sake: 
how can one characterise the sets that support countable ranks? 
It is easy to see that a $\Pi^1_1$ sets admits a countable $\Pi^1_1$-rank if and only if it is Borel (see Lemma \ref{Borel Pi11 sets and countable ranks}). 
The next result (see Theorem \ref{characterisation countable ranks}) shows that even very simple $\Sigma^1_2$ Borel sets need not admit countable $\Sigma^1_2$ ranks. 

\begin{theorem} 
If $\Sigma^1_3$ Cohen absoluteness holds,\footnote{I.e. $V\prec_{\Sigma^1_3} V[G]$ for any Cohen generic extension $V[G]$ of $V$.} then the following conditions are equivalent for any countable $\Pi^1_2$ set $A$. 
\begin{enumerate-(a)} 
\item 
$A\in L_\tau$. 
\item 
\label{intro char countable ranks 2} 
$A$ is a subset of some countable $\Sigma^1_2$-set. 
\item 
\label{intro char countable ranks 3} 
$A$ is a subset of some countable $\Delta^1_2$-set. 
\item 
\label{intro char countable ranks 4} 
There exists a countable $\Sigma^1_2$-rank on the complement of $A$. 
\end{enumerate-(a)} 
If $\omega_1$ is inaccessible in $L$, then the implication from \ref{intro char countable ranks 4} to \ref{intro char countable ranks 2} and \ref{intro char countable ranks 3} holds. 
In particular, the complement of the singleton $0^\#$ does not admit a countable $\Sigma^1_2$-rank. 
\end{theorem} 

Furthermore, any $\Pi^1_1$ Borel set has a \borel code in $L$ by Proposition \ref{sup of Borel ranks of Pi11 sets}. 
(\borel codes are Borel codes on countable ordinals; in $V$ these notions are equivalent.)
This suggests to study $\Sigma^1_2$ sets Borel sets with \borel codes in $L$. 
We do this in Section \ref{section Sigma12 Borel sets}. 
This continues the study of $\Pi^1_1$ Borel sets begun by Kechris, Marker and Sami \cite{MR1011178}. 
It leads us to the open problem whether all absolutely $\Delta^1_2$ Borel sets have \borel codes in $L$. 
A positive level-by-level solution of this problem, analogous to Louveau's separation theorem, would strengthen several classical theorems of descriptive set theory such as Shoenfield absoluteness and the Mansfield-Solovay theorem as well as a recent result of Kanovei and Lyubetsky \cite{kanovei2019borel}. 
We prove a new partial result towards this problem in Theorem \ref{proper forcing no new sets}: 

\begin{theorem} 
Proper forcing does not add new absolutely $\Delta^1_2$ Borel sets. 
\end{theorem}

In particular, in any extension of $L$ by proper forcing, any absolutely $\Delta^1_2$ Borel set has a Borel code in $L$.

\section{Preliminaries}

\subsection{Notation} 

We write $\forall \alpha, \beta\dots$ and $\exists \alpha, \beta,\dots$ for quantifiers ranging over ordinals. 
We say that a set $x$ is \emph{$\Phi$-definable} with respect to a class  $\Phi$ of formulas if there is a formula $\varphi(u)\in \Phi$ such that $x$ is the unique set with $\varphi(x)$. 
We write $M\prec_n N$ if $M$ is a $\Sigma_n$-elementary substructure of $N$. 
Let $p\colon \Ord\times \Ord\rightarrow \Ord$ denote a standard pairing function.\footnote{See \cite[Chapter 3]{Je03}.}  
Let $\WO$ denote the set of wellorders on $\omega$, i.e. the set of $p[x]$ where $x$ is a wellorder on $\omega$. 
Moreover, let $\WO_\alpha$ denote the set of $x\in\WO$ with order type $\alpha$ and $\WO_{{\leq}\alpha}$ the set of those with order type at most $\alpha$. 
Let $\alpha_x=\otp(x)$ denote the order type of any $x\in \WO$. 
Let $x_\alpha$ denote the $<_L$-least real coding an ordinal $\alpha<\omega_1^L$. 
Let $\alpha^\oplus$ denote the least admissible ordinal strictly above $\alpha$. 
The \emph{$L$-rank} of a set $x\in L$ is the least $\alpha$ with $x\in L_{\alpha+1}$. 
Call an ordinal $\beta$ an \emph{$\alpha$-index} if $\beta>\alpha$ and some $\Sigma_1^{L_{\omega_1}}$ fact with parameters ${\leq}\alpha$ first becomes true in $L_\beta$. 
An \emph{index} is an $\emptyset$-index.\footnote{Note that the notion \emph{index} is sometimes used in a different sense for those $L$-levels where a new real appears.}

\subsection{Borel sets and codes} 

For any topological space $X$, the Borel sets are generated from the open sets by forming complements, countable unions and intersections. 
For any subset $A$ of $X$, the Borel subsets of $A$ are precisely the sets of the form $A\cap B$, where $B$ is a Borel subset of $X$. 
Suppose that $\Gamma$ is a class of of subsets of $X$. 
Sets of the form $A\cap B$ with $B\in \Gamma$ are called \emph{relative $\Gamma$} subsets of $A$. 

A \emph{\borelomega-code} for a Borel subset of $2^\omega$ is a subset of $\omega$ that codes a Borel set as in \cite[Section 25]{Je03}. 
Suppose that $\alpha$ is a multiplicatively closed ordinal. 
Then $\alpha$ is closed under the pairing function $p\colon \Ord \times \Ord \rightarrow \Ord$. 
A \emph{\borelalpha-code} for a subset of $2^\omega$ is defined as in \cite[Section 25]{Je03} by replacing $\omega$ with $\alpha$. 
If $\gamma$ is a limit of multiplicatively closed ordinals, then a \emph{\borellessgamma-code} is by definition a \borelalpha-code for some multiplicatively closed $\alpha<\gamma$. 
A \emph{\borelinfty-code} is a \borelalpha-code for some ordinal $\alpha$. 
We similarly define codes for ${\bf \Sigma}^0_\alpha$ and ${\bf \Pi}^0_\alpha$ sets. 
Using \borelgamma- and \borellessgamma-codes, we can define \emph{${\bf \Sigma}_\alpha^{(\gamma)}$-} and \emph{${\bf \Sigma}_\alpha^{({<}\gamma)}$-codes} for multiplicatively closed ordinals $\gamma$. 
For each notion of code, we define a type of set, for instance a set defined by a ${\bf \Sigma}_\alpha^{(\gamma)}$-code is called a \emph{${\bf \Sigma}_\alpha^{(\gamma)}$-set}.

\borel-codes are relevant in Section \ref{section Sigma12 Borel sets}. 
They are equivalent to \borelomega-codes in the sense that they code the same sets, but a set may have a \borel-code in $L$ while it does not have a \borelomega-code in $L$.

\subsection{$\Sigma_2$-admissible sets}\footnote{A formula is $\Sigma_n$ if it is built from a $\Sigma_0$-formula, i.e. one with only bounded quantifiers, by adding $n$ alternating blocks of quantifiers of the form $\exists x_0 \dots \exists x_k$ and $\forall x_0 \dots \forall x_k$ in front.} 
An \emph{admissible set} is a transitive model of Kripke-Platek set theory {\sf KP} with infinity. 
An admissible set is called \emph{$\Sigma_2$-admissible} if it is a model of $\Sigma_2$-collection. 
It is easy to show that any $\Sigma_2$-admissible set satisfies $\Sigma_2$-replacement and $\Delta_2$-separation.\footnote{Some other facts about $\Sigma_n$-admissible sets can be found in \cite[Section 1]{kranakis1982reflection}.}

\begin{lemma} 
\label{substructures of Sigma2-admissibles} 
Suppose that $M$ is transitive, $N$ is $\Sigma_2$-admissible and $M\prec_2 N$. 
Then $M$ is $\Sigma_2$-admissible. 
\end{lemma} 
\begin{proof} 
It is easy to check that $M$ is admissible. 
Note that $\Pi_1$-collection implies $\Sigma_2$-collection. 
Towards a contradiction, suppose this fails in $M$. 
Then there exists a set $A\in M$ and a $\Pi_1$-definable relation $R\subseteq A\times M$ over $M$ that is total on $A$ with 
$$ M \models \forall B\ \exists a\in A\ \forall b\in B\ \neg (a,b)\in R.$$ 
$\Sigma_1$-definable sets are closed under bounded quantification over all admissible sets and the translation is uniform. 
Therefore, the previous statement is $\Pi_2$ over $M$. 
Since $M\prec_2 N$, this holds in $N$. 
Since $M$ is transitive, $N$ believes that $R$ is a total relation on $A$. 
But this contradicts the assumption that $N$ is $\Sigma_2$-admissible. 
\end{proof}

\subsection{Infinite time computation} 
\label{section - infinite time} 

We shall motivate the definition of ranks by a natural example: ranks induced by infinite time Turing machines (ittm's). 
We first give a brief sketch of these machines (for more details see \cite{hamkins2000infinite}.) 
An ittm-program is a Turing program with states $0, \dots, n$ for some natural $n$.  
The hardware of an ittm consists of an input, work and output tape, each of length $\omega$. 
Each cell contains $0$ or $1$. 
While an ittm works like a Turing machine at successor stages, at limit stages the head is set to the leftmost cell, the state to the limit inferior of the earlier states,\footnote{The definition of the limit state is different from \cite{hamkins2000infinite}, but the models are computationally equivalent.} 
and the contents of each cell to the inferior limit of the earlier contents of this cell. 
An ordinal Turing machine (otm) works similarly with an ordinal length tape \cite{koepke2005}. 
A set $A$ of reals is called \emph{ittm-semidecidable} if there is an ittm-program $p$ such that the computation $p(x)$ with input $x$ halts if and only if $x\in A$. 
$A$ is called \emph{ittm-decidable} if the set and its complements are both ittm-semidecidable. 
The definitions for otm's are analogous. 


Next is an important structural property of the class of semidecidable sets with respect to various machine models such as ittm's. 
In fact, the existence of any rank suffices to prove this property. 
For a set $A$ semi-decided by an ittm algorithm $p$, the rank of a real $x$ is defined as the halting time of $p(x)$. 

\begin{remark} 
\label{reduction for ittms} 
The \emph{reduction property} of a class $\Gamma$ of subsets of $2^\omega$ states that the union of two sets $A$ and $B$ in $\Gamma$ can be partitioned into two disjoint sets $A'\subseteq A$ and $B'\subseteq B$ in $\Gamma$. 
To show this for the class $\Gamma$ of ittm-semidecidable sets, take programs $p$ and $q$ that semidecide $A$ and $B$, respectively, 
and run them synchronously with input $x$. 
Let $x\in A'$ if $p(x)$ halts and $q(x)$ halts at the same or some later time, or $p(x)$ halts and $q(x)$ diverges. 
Similarly, let $x\in B'$ if $q(x)$ halts and $p(x)$ halts at some later time or diverges.\footnote{For a set-up closer to the general argument for ranks below, take a single program $r$ that semidecides the disjoint union $(0^\smallfrown A) \cup (1^\smallfrown B)$ and let $p(x)=r(0^\smallfrown x)$ and $q(1)=r(1^\smallfrown x)$ above.} 
\end{remark} 



\subsection{Ranks} 
\label{section ranks} 
The definition of ranks generalises properties of the class of sets semidecidable by infinite computations. 
Recall that a \emph{prewellorder} $\leq$ on a set $A$ is a wellfounded linear quasiorder on $A$. 
We write $<$ for its strict part.\footnote{Recall that a \emph{quasiorder} is a reflexive transitive relation. 
The \emph{strict part} $<$ of a linear quasiorder $\leq$ on a set $A$ is defined by $x<y:\Leftrightarrow  y\not\leq x$ for $x,y\in A$. 
The \emph{non-strict part} $\leq$ of a strict linear quasiorder $<$ on a set $A$ is defined by $x\leq y:\Leftrightarrow y\not< x$ for $x,y\in A$. 
The notions are dual: every strict linear quasiorder is the strict part of a linear quasiorder, and every linear quasiorder is the non-strict part of a strict linear quasiorder. 
They induce an equivalence relation on the underlying set $A$ defined by $x\equiv y :\Leftrightarrow x\leq y \wedge y \leq x$ and a (strict) linear order on the quotient.}  
A \emph{strict prewellorder} is a wellfounded strict linear quasiorder on $A$. 
A \emph{rank function} for a strict prewellorder $<$ on a set $A$ is a function $f$ from $A$ to the ordinals with $x<y \Leftrightarrow f(x)<f(y)$.   

$\Gamma$ always denotes a collection of subsets of $2^\omega$ that contains all basic open sets and is closed under finite products, finite unions and computable preimages. 
Let $\bar{A}$ denote the complement of a set $A$ and $\bar{\Gamma}=\{ \bar{A}\mid A\in \Gamma\}$. 
A set $A$ is called \emph{$\Gamma$-complete} if it is in $\Gamma$ and every $\Gamma$ set is a preimage of $A$ under a computable function. 
We call a set a \emph{true} $\Gamma$ set if it is in $\Gamma$, but not a continuous preimage of a set in $\bar{\Gamma}$.\footnote{Note that every true $\Gamma$ set is $\Gamma$-complete with respect to continuous functions, if we assume determinacy for Boolean combinations of sets in $\Gamma$, since then Wadge's lemma holds for $\Gamma$.}  
For instance, any $\Gamma$-complete set $A$ is a true $\Gamma$ set. 
To see this, let $\vec{f}=\langle f_x \mid x\in 2^\omega \rangle$ list all continuous functions $f\colon 2^\omega \rightarrow 2^\omega$ in a computable way.\footnote{I.e., the preimages under $f_x$ of basic open set are given by a computable function in $x$.}  
If $A$ were not a true $\Gamma$ set, then $\{x\in 2^\omega \mid  x\notin f_x^{-1}(A) \}=g^{-1}(A)$ for some continuous function $g$, since the function sending $(x,y)$ to $f_x(y)$ is computable. 
We obtain a contradiction for any $x$ with $f_x=g$. 

In the next definition we write $A_y=\{ x \mid (x,y)\in A\}$ for the section of a subset $A$ of $(2^\omega)^2$ at $y\in 2^\omega$. 

\begin{definition} 
\label{definition rank}\footnote{While the definition of ranks arises naturally from computations, one can also argue that it is optimal from the viewpoint of descriptive set theory. 
Overspill is essential, since otherwise virtually any class $\Gamma$ would have the rank property as witnessed by the trivial prewellorder with a single class. 
Moreover, the complexity of the relations $\sqsubseteq$ and $\sqsubset$ is optimal, since they cannot be chosen to Borel for any $\Pi^1_1$ rank on a complete $\Pi^1_1$ set $A$. 
To see this, suppose otherwise and take a continuous reduction of $\WO$ to $A$. 
The image of each $\WO_\alpha$ is bounded in the rank on $A$ by the boundedness lemma \cite[4C.11]{Mo09}. 
However, every bounded initial segment of $A$ with respect to the rank has Borel rank below a fixed countable ordinal. 
So the same holds for the sets $\WO_{{\leq}\alpha}$. 
But $\WO_{\leq\alpha}$ is not $\Pi^0_{2\cdot\beta}$ if $\alpha=\omega^\beta$ (see \cite{sternborelrank} and \cite[Lemma 1.3]{MR1011178}).}
A \emph{$\Gamma$-rank} on a set $A\in \Gamma$ is a prewellorder $\leq$ on $A$ with strict part $<$ such that there exist $\Gamma$  relations $\sqsubseteq$ and $\sqsubset$ with the following properties: 
\begin{enumerate-(1)} 
\item 
\emph{(Left agreement)}\footnote{The relations $\sqsubseteq$ and $\sqsubset$ can be chosen to be transitive by Remark \ref{canonical choice for ranks}. If the rank has limit length, one can then replace left agreement by the equivalent simpler condition that when restricted to $A$, $\sqsubseteq$ equals $\leq$ and $\sqsubset$ equals $<$. 
Although all interesting ranks have limit length, it is useful to allow the case of successor length for more generality.} 
For all $x\in A$: 
\begin{enumerate-(a)} 
\item 
$\sqsubseteq_x$ equals $\leq_x$. 
\item 
$\sqsubset_x$ equals $<_x$. 
\end{enumerate-(a)} 
\item 
\emph{(Overspill)} 
$x\sqsubseteq y$ and $x\sqsubset y$ for all $(x,y)\in A \times (2^\omega \setminus A)$.
\end{enumerate-(1)} 
A fixed strict order preserving function from $(A,<)$ to the ordinals (a \emph{rank function}) is regarded as part of the rank, but 
we omit it when it is not relevant.\footnote{Even for natural ranks defined by infinite computations, the range may fail to be an ordinal.} 
The order type of $\leq$ is called the rank's \emph{length} and is written as $|{\leq}|$. 
Finally, $\Gamma$ has the \emph{rank property} by definition if every set in $\Gamma$ admits a $\Gamma$-rank. 
\end{definition} 

For example, we mentioned the $\Pi^1_1$ set $\WO$ of wellorders on $\omega$ with the Borel layers $\WO_\alpha$, where $\WO_\alpha$ consists of all wellorders of order type $\alpha$. 
It is easy to see that the relations $\leq$ and $<$ on $\WO$ given by the order types form a $\Pi^1_1$ rank. 
Since $\WO$ is $\Pi^1_1$-complete, it follows that $\Pi^1_1$ has the rank property \cite[Theorem 4B.2]{Mo09}. 
Using projections, it can then be easily shown that $\Sigma^1_2$ has the rank property as well \cite[Theorem 4B.3]{Mo09}.  
%

\begin{remark} 
\label{canonical choice for ranks} 
For any $\Gamma$-rank given by $\leq$, there is a canonical choice of relations $\sqsubseteq$ and $\sqsubset$. 
The next diagram defines $\sqsubseteq$ in cases depending on membership in $A$ of $x$ and $y$. 
We then define $\sqsubset$ similarly by replacing $\leq$ with $<$. 

\begin{figure}[H] 
\begin{tabular}{ l | l | l | } 
 & $y\in A$ & $y\notin A$ \\ 
 \hline 
$x\in A$ & $x\leq y$ & {\tt true} \\ 
\hline 
$x\notin A$ & {\tt false} & {\tt false} \\ 
\hline 
\end{tabular} 
\end{figure}
If $\unlhd$ and $\lhd$ are arbitrary witnesses for the $\Gamma$-rank, then 
$x \sqsubseteq y \Leftrightarrow x\in A\wedge x\unlhd y$ and $x \sqsubset y \Leftrightarrow x\in A\wedge x\lhd y$. 
Therefore $\sqsubseteq$ and $\sqsubset$ are in $\Gamma$ and thus these relations are the required witnesses for the $\Gamma$-rank. 

It follows that the relations $\sqsubseteq$ and $\sqsubset$ can be chosen transitive and such that $\sqsubset$ is a subset of $\sqsubseteq$. 
Moreover, the prewellorder extending $\leq$ by adding the complement of $A$ a single equivalence class on top is in  $\bar{\Gamma}$; this prewellorder is the complement of $\sqsubset$. 
\end{remark} 

\begin{remark} \ 
\begin{enumerate-(1)} 
\item 
Every class $\Gamma$ with the rank property satisfies the reduction property. 
To see this, suppose that $A$ and $B$ are sets in $\Gamma$ and argue as in Remark \ref{reduction for ittms} using a $\Gamma$-rank on the disjoint union of $A$ and $B$. 
\item 
Any $\Gamma$-rank on a true $\Gamma$ set $A$ has limit length. 
If it had successor length, let $y$ be of maximal rank in $A$. 
Then $A=\{x\mid x\sqsubseteq y\}$ is a continuous preimage of a set in $\Gamma$, contradicting the assumption. 
\end{enumerate-(1)} 
\end{remark}




\section{Robust ordinals}

In proof theory, ordinals $\alpha$ that are $\Sigma_n$-stable in $\beta$, i.e. with $L_\alpha\prec_n L_\beta$, are studied \cite[Section 5]{rathjen2017higher}. 
We work with a variant of this definition where elementarity is relaxed. 

\subsection{Kechris' ordinal}
\label{subsection Kechris ordinal} 

By a \emph{formula in $\alpha$}, we mean a formula with parameters ${\leq}\alpha$. 
Suppose that $M\subseteq N$ are transitive sets. 
We write $\vec{x}$ for finite tuples of variables. 
We say that $(M,N)$ satisfies \emph{Tarski's test} for $\Sigma_{n+1}$-formulas in $\alpha$ if for every formula $\exists \vec{x}\ \varphi(\vec{x},y)$ in $\alpha$ true in $N$, where $\varphi$ is $\Pi^1_n$ and $y\in M$, 
there is some $\vec{x}\in M$ such that $\varphi(\vec{x},y)$ holds in $N$. 

\begin{definition} 
\label{definition of tau} 
Suppose that $\alpha$ is a countable ordinal with $\alpha\geq\omega$ and $n\geq1$. 
We call an ordinal $\gamma$ 
\emph{$\Sigma_n^{\alpha}$-robust} if $(L_\gamma,L_{\omega_1})$\footnote{As mentioned above, $\omega_1$ always denotes $\omega_1^L$.}  satisfies Tarki's test for $\Sigma_n$-formulas in $\alpha$. 
Let $\sigma_n^{\alpha}$ denote the least $\Sigma_n^{\alpha}$-robust ordinal. 
\end{definition} 

We shall omit $\alpha$ when $\alpha=\omega$. 
Following the notation of \cite{decisiontimes}, we shall write $\sigma=\sigma_1$ for the first $\Sigma_1$-robust ordinal and $\tau=\sigma_2$ for the first $\Sigma_2$-robust ordinal. 
We further write $\sigma^\alpha=\sigma_1^\alpha$ and $\tau^\alpha=\sigma_2^\alpha$.  

It is easy to see that any set $x\in L_{\omega_1}$ that is $\Sigma_{n}$-definable over $L_{\omega_1}$ in $\alpha$ is an element of $L_{\sigma_n^{\alpha}}$.  
The converse is not necessarily true.\footnote{
However, it is true for $n=1$. 
To see this, it suffices to check that every element of $L_\gamma$ is $\Sigma_1$-definable in $\alpha$ over $L_{\omega_1}$, where $\gamma>\alpha$ is least with $L_\gamma\prec_1 L_{\omega_1}$. 
So let $M$ denote the set of elements of $L_{\omega_1}$ that are $\Sigma_1$-definable in $\alpha$. 
Clearly $M\prec_1 L_{\omega_1}$. 
If the claim fails, then the collapsing map $\pi\colon M\rightarrow \bar{M}$ moves some $x\in M$. 
We can take $x$ with minimal $L$-rank. 
Then $x$ is definable by a $\Sigma_1$-formula $\varphi(x)$ in $\alpha$ over $L_{\omega_1}$. 
Since $\pi$ fixes all ordinals ${\leq}\alpha$ pointwise and $\bar{M}$ is transitive, $\varphi(\pi(x))$ holds in $L_{\omega_1}$ by $\Sigma_1$ upwards absoluteness. 
But $\pi(x)\neq x$ contradicts the uniqueness of $x$.}
For instance, if $\alpha\geq \omega$ is a countable ordinal (in $V$) with $\alpha^{+L}<\omega_1$, then not every ordinal below $\tau^\alpha$ is $\Sigma_2$-definable in $\alpha$ over $L_{\omega_1}$, since there are only $|\alpha|^L$ many such ordinals while $\tau^{\alpha}>\alpha^{+L}$. 

However, we claim that $\sigma_n^{\alpha}$ equals the \emph{supremum} of those ordinals which are $\Sigma_n$-definable in $\alpha$ over $L_{\omega_1}$. 
The same holds for ordinals $\Pi_{n-1}$-definable in $\alpha$ over $L_{\omega_1}$ for all $n\geq 2$. 
Note that the suprema for $\Sigma_n$ and $\Pi_{n-1}$ are equal by \cite[Lemma 4.1]{decisiontimes} for $n=2$. 
An analogous proof works for all $n\geq 2$. 
Returning to the claim, we want to see that the ordinals $\Sigma_n$-definable over $L_{\omega_1}$ in $\alpha$ are unbounded in $\sigma^\alpha_n$. 
Suppose that $\psi(y)=\exists \vec{x}\ \varphi(\vec{x},y)$ is a $\Sigma_{n}$-formula with parameters ${\leq}\alpha$ that is true in $L_{\omega_1}$, where $\varphi(\vec{x},y)$ is a $\Pi_{n-1}$-formula. 
We can assume $\alpha=0$, since the general case is similar. 
It suffices to show that $\psi(y)$ has a $\Sigma_{n}$-definable witness. 
To see this, note that the $L$-least solution $(\vec{x},y)$ of $\varphi$ is $\Sigma_{n}$-definable. 
It is the unique $(\vec{x},y)$ such that $\varphi(\vec{x},y)$ holds and for some $L_\alpha$ containing $(\vec{x},y)$, all other tuples in $L_\alpha$ strictly $L$-below $(\vec{x},y)$ fail to satisfy $\varphi$ in $L_\alpha$. 
Thus $y$ is a $\Sigma_{n}$-definable solution of $\psi$. 

We now relate $\tau$ with the ordinal $\gamma^1_2$.\footnote{According to \cite{MR1011178}, the ordinals $\gamma^1_n$ were first studied by Kechris.} 
The latter plays a role in 
the analysis of Borel $\Pi^1_1$ sets in \cite{MR1011178}. 
Let 
$$\gamma^1_2:=\sup\{\min(A)\mid \emptyset\neq A\subseteq \omega_1,\ \otp^{-1}(A)\in \Pi^1_2 \}$$ 
$$\gamma:=\sup\{\min(A)\mid \emptyset\neq A\subseteq \omega_1,\ A\in \Pi_1^{H_{\omega_1}}\}$$ 

\begin{proposition} 
$\gamma^1_2=\gamma=\tau$. 
\end{proposition} 
\begin{proof} 
Recall that $\Sigma_1$ and $\Pi_1$-formulas are absolute between $H_{\omega_1}$ and $V$. 
Moreover, any $\Pi^1_2$-definable set of reals is $\Pi_1$-definable and conversely. 

\begin{claim*} 
$\gamma^1_2=\gamma$. 
\end{claim*} 
\begin{proof} 
$\gamma^1_2\leq\gamma$: 
Suppose that $A$ is a nonempty set of countable ordinals such that $\pi^{-1}(A)$ is $\Pi^1_2$-definable. 
We want to show that $\min(A)\leq \gamma$. 
We have that $\pi^{-1}(A)$ is defined by some $\Pi_1$-formula $\varphi(x)$. 
$$A=\{\alpha\mid \forall x\ [(\omega,x)\cong(\alpha,<) \rightarrow \varphi(x)] \}$$ 
is a $\Pi_1$-definition of $A$, since the existence of isomorphisms is expressible by a $\Sigma^1_1$-formula. 
By the definition of $\gamma$, we have $\min(A)\leq\gamma$. 

$\gamma\leq\gamma^1_2$: 
Suppose that $A$ is a subset of $\omega_1$ that is $\Pi_1$-definable over $H_{\omega_1}$ by some formula $\varphi(x)$. 

$$\pi^{-1}(A)=\{x\in \WO \mid \forall \alpha\ [(\omega,x)\cong(\alpha,<) \rightarrow \varphi(\alpha)] \}$$ 
is a $\Pi_1$-definition of $\pi^{-1}(A)$. 
Hence $\pi^{-1}(A)$ is $\Pi^1_2$-definable. 
By the definition of $\gamma^1_2$, we have $\min(A)\leq\gamma^1_2$. 
\end{proof} 

In order to reflect a $\Pi_1$-formula over $H_{\omega_1}$ below $\tau$, we will rewrite it as a $\Pi_1$-formula over $L_{\omega_1}$. 
For any $\Pi_1$-formula $\varphi(\alpha)$, all transitive models $M\subseteq N$ of $\ZFC$ of height at least $\omega_1$ and any ordinal $\alpha$ that is countable in $M$, we have $M\models \varphi(\alpha) \Longleftrightarrow N \models \varphi(\alpha)$ by Shoenfield absoluteness. 

\begin{claim*} 
For any $\Pi_1$-formula $\varphi(x)$ and any countable ordinal $\alpha$, $H_{\omega_1} \models \varphi(\alpha) \Longleftrightarrow \Vdash^{L_{\omega_1}}_{\Col(\omega,\alpha)} \varphi(\alpha)$.\footnote{$\Col(\omega,\alpha)$ denotes the standard collapse to make $\alpha$ countable.}  
\end{claim*} 
\begin{proof} 
Since $\Col(\omega,\alpha)$ is homogeneous, it suffices to prove absoluteness of $\varphi(\alpha)$ between $H_{\omega_1}$ and any $\Col(\omega,\alpha)$-generic extension $L_{\omega_1}[G]$ of $L_{\omega_1}$. 
Since $\Col(\omega,\alpha)$ is homogeneous, we can pick any such filter. 
We thus assume that $G$ is $\Col(\omega,\alpha)$-generic over $V$. 
Since $\alpha$ is countable in both $V$ and $L[G]$, we have 
$$ H_{\omega_1} \models \varphi(\alpha) \Longleftrightarrow V\models \varphi(\alpha) \Longleftrightarrow  V[G]\models \varphi(\alpha) \Longleftrightarrow  L[G]\models \varphi(\alpha) \Longleftrightarrow L_{\omega_1}[G] \models \varphi(\alpha)$$ 
by Shoenfield absoluteness. 
%
\end{proof} 

Suppose that $\alpha$ is a countable ordinal. 
For any $\Sigma_1$-formula, the formula $\Vdash^{L_{\omega_1}}_{\Col(\omega,\alpha)} \varphi(\alpha)$ is a $\Sigma_1$-formula as well. 
Similarly, for any $\Pi_1$-formula, the formula $\Vdash^{L_{\omega_1}}_{\Col(\omega,\alpha)} \varphi(\alpha)$ is a $\Pi_1$-formula. 

\begin{claim*} 
$\gamma=\tau$. 
\end{claim*} 
\begin{proof} 
$\gamma\leq\tau$: 
Suppose that $A$ is a subset of $\omega_1$ that is $\Pi_1$-definable over $H_{\omega_1}$ by some formula $\varphi(x)$. 
We want to show $\min(A)\leq \tau$. 
By the previous claim, the formula $\exists\alpha\ \Vdash^{L_{\omega_1}}_{\Col(\omega,\alpha)} \varphi(\alpha)$ holds in $L_{\omega_1}$. 
By Tarki's test for $\Sigma_2$-formulas, there is some $\alpha<\tau$ such that  $\Vdash^{L_{\omega_1}}_{\Col(\omega,\alpha)} \varphi(\alpha)$ holds in $L_{\omega_1}$. 
By the previous claim, $\varphi(\alpha)$ holds in $H_{\omega_1}$ and thus $\alpha\in A$. 
Since $\alpha<\tau$, $\min(A)<\tau$. 

$\tau\leq\gamma$: 
Suppose that $\alpha$ is $\Sigma_2$-definable over $L_{\omega_1}$. 
We want to show $\alpha\leq\gamma$. 
Take a formula $\exists \vec{x}\ \psi(\vec{x},y)$ defining $\alpha$, where $\psi(\vec{x},y)$ is a $\Pi_1$-formula. 
We can rewrite $\exists \vec{x}\ \psi(\vec{x},y)$ as the formula $\exists \beta \ \varphi(\beta,y)$, where $\varphi(\beta,y)$ states that for every transitive model $M$ of $V=L$ of height $\beta$, there is some $\vec{x}\in M$ such that $\psi(\vec{x},y)$ holds. 
$\varphi$ is a $\Pi_1$-formula. 
Note that the pairing function $p\colon \Ord\times\Ord \rightarrow \Ord$ is $\Delta_1$-definable and satisfies $p(\epsilon,\eta) \geq \epsilon,\eta$ for all $\epsilon,\eta$. 
The set $A=\{p(\epsilon,\eta) \mid \varphi(\epsilon,\eta)\}$ is $\Pi_1$-definable over $H_{\omega_1}$. 
Since $\eta=\alpha$ for all $\epsilon,\eta$ with $p(\epsilon,\eta) \in A$, we have $p(\epsilon,\eta) \geq\eta=\alpha$. 
Hence $\alpha\leq\min(A)\leq\gamma$. 
\end{proof} 
The previous claims yield $\gamma^1_2=\gamma=\tau$ as required. 
\end{proof} 

For this first part of Theorem \ref{main theorem}, it remains to show that $\tau$ equals the supremum of least elements of nonempty $\Sigma_2$-definable subsets of $\omega_1$ over $L_{\omega_1}$. 
To see this, suppose that $\exists x\ \varphi(\alpha,x)$ defines a nonempty set $A$ of ordinals over $L_{\omega_1}$, where $\varphi$ is a $\Pi_1$-formula. 
Then 
$$ B:= \{ p(\alpha,\beta)\mid \exists x\in L_\beta\ \varphi(\alpha,x) \}  $$ 
is a nonempty set $\Pi_1$-definable over $H_{\omega_1}$ with $\min(A)\leq \min(B)$.

\subsection{Robust versus stable ordinals}
\label{section robust vs stable} 
While it is clear that $L_{\sigma^\alpha}\prec_1 L_{\omega_1}$, a similar fact can be easily seen for $\tau^\alpha$ as well: 

\begin{lemma} 
\label{Sigma1 elementarity of Ltau} 
$\tau^\alpha$ is closed under function that sends $\beta$ to $\sigma^\beta$; 
in particular, $L_{\tau^\alpha}\prec_1 L_{\omega_1}$. 
\end{lemma} 
\begin{proof} 
Since $\sigma_\beta\leq \sigma_\gamma$ holds for all $\beta\leq\gamma$, we can assume that $\beta$ is $\Pi_1$-definable over $L_{\omega_1}$. 
Note that $\sigma_\beta$ is $\Pi_1$-definable from $\beta$. 
It follows that $\sigma_\beta$ is $\Sigma_2$-definable in $\alpha$, so we obtain $\sigma_\beta<\tau^\alpha$, as required. 
\end{proof} 

Note that the previous lemma implies that $L_\tau$ is closed under $\Sigma_2$-recursion in $L_{\omega_1}$ (with respect to a total recursion rule) along ordinals. 
To see this, we argue that for each $\alpha<\tau$, the recursion yields a function $f_\alpha \colon \alpha \rightarrow L_\tau$ in $L_\tau$. 
It suffices to show this for unboundedly many $\alpha<\tau$, so we assume that $\alpha$ is $\Pi_1$-definable over $L_{\omega_1}$. 
The existence of a function $f\colon \alpha \rightarrow V$ satisfying the recursion is $\Sigma_2$ over $L_{\omega_1}$ and hence the definition of $\tau$ yields such a function in $L_\tau$. 
Using Lemma \ref{Sigma1 elementarity of Ltau}, it follows that the function $f_\tau\colon \tau \rightarrow L_\tau$ defined by the recursion is $\Sigma_2$-definable over $L_\tau$. 

However, $L_{\tau^\alpha}\prec_2 L_{\omega_1}$ is not true in general by the next result. 
The lemma also analyses the 
least ordinal $\hat{\tau}^\alpha$ such that $L_{\hat{\tau}^\alpha}$ and $L_{\omega_1}$ satisfy the same $\Sigma_2$ sentences in $\alpha$. 
In general, $\hat{\tau}^\alpha\leq\tau^\alpha$ since every $\Sigma_2$-statement in $\alpha$ that holds in $L_{\omega_1}$ also holds in $L_{\tau^\alpha}$. 
In detail, suppose that $\theta=\exists \vec{x}\ \psi(\vec{x})$ holds in $L_{\omega_1}$, where $\psi$ is a $\Pi_1$-formula in $\alpha$. 
There is some $\vec{x}\in L_{\tau^\alpha}$ such that $\psi(\vec{x})$ holds in $L_{\omega_1}$ by the definition of $\tau^\alpha$. 
Since $\psi$ is $\Pi_1$, $\psi(\vec{x})$ also holds in $L_{\tau^\alpha}$ by Lemma \ref{Sigma1 elementarity of Ltau}, so $\theta$ holds in $L_{\tau^\alpha}$, as required. 

\begin{lemma} 
\label{when the least Sigma2-stable is Sigma2-elementary}  
The following conditions are equivalent for any infinite countable ordinal $\alpha$: 
\begin{enumerate-(a)} 
\item 
\label{Sigma2-stable is elementary 1}  
$\alpha^{+L}=\omega_1$
\item 
\label{Sigma2-stable is elementary 5} 
Every element of $L_{\tau^\alpha}$ is $\Sigma_2$-definable in $\alpha$ over $L_{\omega_1}$. 
\item 
\label{Sigma2-stable is elementary 4}  
$L_{\tau^\alpha}\prec_2 L_{\omega_1}$
\item 
\label{Sigma2-stable is elementary 6}  
$L_{\tau^\alpha}$ is $\Sigma_2$-admissible. 
\item 
\label{Sigma2-stable is elementary 2}  
$\hat{\tau}^\alpha=\tau^\alpha$
\item 
\label{Sigma2-stable is elementary 3} 
$L_{\hat{\tau}^\alpha}\prec_1 L_{\omega_1}$ 
\end{enumerate-(a)} 
\end{lemma} 
\begin{proof} 
\ref{Sigma2-stable is elementary 1} $\Rightarrow$ \ref{Sigma2-stable is elementary 5}: 
Suppose that $x\in L_{\tau^\alpha}$. 
Take any $\gamma$ that is $\Sigma_2$-definable in $\alpha$ over $L_{\omega_1}$ with $x\in L_\gamma$. 
Since $\alpha^{+L}=\omega_1$, $|L_\gamma|\leq|\alpha|$ in $L$ and the $L$-least surjection $f\colon \alpha\rightarrow L_\gamma$ is $\Sigma_2$-definable in $\alpha$ over $L_{\omega_1}$. 
Hence every element of $L_\gamma$ is $\Sigma_2$-definable in $\alpha$ over $L_{\omega_1}$. 

\ref{Sigma2-stable is elementary 5} $\Rightarrow$ \ref{Sigma2-stable is elementary 4}: 
By \ref{Sigma2-stable is elementary 5}, $L_{\tau^\alpha}$ equals the collection of all sets which are $\Sigma_2$-definable in $\alpha$ over $L_{\omega_1}$. 
Suppose that $\exists \vec{x}\ \varphi(\vec{x},y)$ holds in $L_{\omega_1}$ for some $y\in L_{\tau^\alpha}$, where $\varphi$ is a $\Pi_1$-formula in $\alpha$. 
Then the least $\xi$ such that there is some $\vec{x}\in L_\xi$ with $L_{\omega_1}\models \varphi(\vec{x},y)$ is $\Sigma_2$-definable in $\alpha$ over $L_{\omega_1}$. 
Hence $\xi<\tau^\alpha$. 
Take any such $\vec{x}\in L_\xi$. 
Then $L_{\tau^\alpha}\models \varphi(\vec{x},y)$ by $\Pi_1$ downwards absoluteness. 

\ref{Sigma2-stable is elementary 4} $\Rightarrow$ \ref{Sigma2-stable is elementary 6}: 
This holds by Lemma \ref{substructures of Sigma2-admissibles}. 

\ref{Sigma2-stable is elementary 6} $\Rightarrow$ \ref{Sigma2-stable is elementary 1}: 
Towards a contradiction, suppose that $\alpha^{+L}<\omega_1$. 
We shall find some $S\in L_{\tau^\alpha}$ and a cofinal function $f\colon S\rightarrow \tau^\alpha$ that is $\Sigma_2$-definable over $L_{\tau^\alpha}$ with parameters. 
Let $\vec{\varphi}=\langle \varphi_i(x,y)\mid i\in\alpha\rangle$ list all $\Pi_1$-formulas with the free variables $x,y$ and parameters ${\leq}\alpha$ in a simply definable way. 
Let $S$ denote the set of all $i\in \alpha$ such that $\exists x\ \varphi_i(x,y)$ defines a singleton over $L_{\omega_1}$. 
We have $\tau^\alpha>\alpha^{+L}$, since $\alpha^{+L}$ is $\Pi_1$-definable in $\alpha$. 
Hence $S\in L_{\tau^\alpha}$ by condensation. 
Define $f\colon S\rightarrow \tau^\alpha$ by letting $f(i)$ be the least $\gamma$ such that there exist $x,y\in L_\gamma$ with $L_{\omega_1}\models \varphi_i(x,y)$. 
$f$ is cofinal in $\tau^\alpha$, since $\tau^\alpha$ is the supremum of ordinals which are $\Sigma_2$-definable over $L_{\omega_1}$. 
$f$ is defined over $L_{\omega_1}$ by a Boolean combination of $\Sigma_1$-formulas with parameters in $\tau^\alpha$. 
Since $L_{\tau^\alpha}\prec_{\Sigma_1}L_{\omega_1}$ by Lemma \ref{Sigma1 elementarity of Ltau}, the same combination defines $f$ over $L_{\tau^\alpha}$. 
In particular, $f$ is $\Sigma_2$-definable over $L_{\tau^\alpha}$ with parameters. 

\ref{Sigma2-stable is elementary 1} $\Rightarrow$ \ref{Sigma2-stable is elementary 2}: 
Since \ref{Sigma2-stable is elementary 1} $\Rightarrow$ \ref{Sigma2-stable is elementary 5}, 
every element of $L_{\tau^\alpha}$ is $\Sigma_2$-definable in $\alpha$ over $L_{\omega_1}$ by \ref{Sigma2-stable is elementary 5}. 

\begin{claim*} 
A $\Sigma_2$-formula $\varphi(x)$ in $\alpha$ defines a singleton over $L_{\tau^\alpha}$ if and only if it defines a singleton over  $L_{\hat{\tau}^\alpha}$. 
\end{claim*} 
\begin{proof} 
Suppose that $\varphi(x)$ defines a singleton in $L_{\tau^\alpha}$. 
Since $\exists x\ \varphi(x)$ and $\forall x, y\ [(\varphi(x) \wedge \varphi(y) \rightarrow x= y]$ hold true in $L_{\tau^\alpha}$, they hold in $L_{\hat{\tau}^\alpha}$ as well. 
The other direction is analogous. 
\end{proof} 

Towards a contradiction, suppose that $\hat{\tau}^\alpha<\tau^\alpha$. 
By \ref{Sigma2-stable is elementary 5}, $\hat{\tau}^\alpha$ is $\Sigma_2$-definable in $\alpha$ over $L_{\tau^\alpha}$ by some formula $\varphi(x)$. 
Let $\psi(x)= \varphi(x)\wedge x\in \mathrm{Ord}$. 
By the previous claim, $\psi$ defines an ordinal $\gamma$ in $L_{\hat{\tau}^\alpha}$. 
It suffices to show that $L_\gamma$ has the same theory with parameters ${\leq}\alpha$ as $L_{\hat{\tau}^\alpha}$, since this contradicts the minimality of $\hat{\tau}^\alpha$. 
To see this, take any sentence $\theta$ in $\alpha$ that is true in $L_{\hat{\tau}^\alpha}$. 
Then $\chi=\exists \eta\  [\varphi(\eta) \wedge \exists y \ (L_\eta=y \wedge \theta^y)]$ holds in $L_{\tau^\alpha}$, where $\theta^y$ denotes the relativisation obtained by replacing all quantifiers $\exists x$ and $\forall x$ by bounded quantifiers $\exists x\in y$ and $\forall x\in y$, respectively. 
$\chi$ is $\Sigma_2$, since the formula $L_\alpha=y$ is $\Sigma_1$. 
Therefore $\chi$ holds in $L_{\hat{\tau}^\alpha}$. 
So $L_\gamma$ satisfies $\theta$. 

\ref{Sigma2-stable is elementary 2} $\Rightarrow$ \ref{Sigma2-stable is elementary 1}: 
Suppose that $\alpha^{+L}<\omega_1$. 
Take some $L_\beta$ of size ${<}\alpha^{+L}$ with the same first-order theory as $L_{\omega_1}$ in parameters ${\leq}\alpha$ by L\"owenheim-Skolem in $L$. 
Then $\hat{\tau}^\alpha \leq\beta < \alpha^{+L}<\tau^\alpha$. 

\ref{Sigma2-stable is elementary 2} $\Rightarrow$ \ref{Sigma2-stable is elementary 3}: This follows from Lemma \ref{Sigma1 elementarity of Ltau}. 

\ref{Sigma2-stable is elementary 3} $\Rightarrow$ \ref{Sigma2-stable is elementary 2}:  
It suffices to show $\tau^\alpha \leq \hat{\tau}^\alpha$. 
Suppose that $\exists \vec{x}\ \varphi(\vec{x})$ 
holds in $L_{\omega_1}$, where $\varphi$ is a $\Pi_1$-formula in $\alpha$. 
Let $\gamma$ be least such that there is some $\vec{x}\in L_\gamma$ with $L_{\omega_1}\models \varphi(\vec{x})$. 
It suffices to show $\gamma\leq\hat{\tau}^\alpha$. 
In other words, it remains to find some $\vec{x}\in L_{\hat{\tau}^\alpha}$ such that $\varphi(\vec{x})$ holds in $L_{\omega_1}$. 
To see this, note that $\exists \vec{x}\ \varphi(\vec{x})$ holds in $L_{\hat{\tau}^\alpha}$ by the definition of $\hat{\tau}^\alpha$. 
Take $\vec{x}\in L_{\hat{\tau}^\alpha}$ such that $\varphi(\vec{x})$ holds in $L_{\hat{\tau}^\alpha}$. 
Then $\varphi(\vec{x})$ holds in $L_{\omega_1}$ by \ref{Sigma2-stable is elementary 3} as required. 
\end{proof} 

Note that  $\hat{\tau}^\alpha>(\tau^\alpha)^L$ if $\alpha<\omega_1^L$ and $\alpha^{+L}<\omega_1$. 
To see this, we first show that $\hat{\tau}^\alpha\geq (\tau^\alpha)^L$. 
Any $\gamma<(\tau^\alpha)^L$ is $\Sigma_2$-definable in $\alpha$ over $L_{\omega_1^L}$ by some $\Sigma_2$-formula $\varphi_\gamma(x)$ in $\alpha$ by Lemma \ref{when the least Sigma2-stable is Sigma2-elementary} \ref{Sigma2-stable is elementary 5}. 
Then $\gamma$ is definable in $L_{\omega_1}$ by $\psi_\gamma(x)=(L_{\omega_1^L}\models \varphi_\gamma(x))$, a $\Sigma_2$-formula in $\alpha$. 
For any $\gamma<(\tau^\alpha)^L$, $\exists x\ \psi_\gamma(x)$ and $\forall x,y\ [(\psi_\gamma(x)\wedge \psi_\gamma(y) )\rightarrow x= y]$ hold in $L_{\omega_1}$. 
Moreover, for any $\gamma<\delta<\tau^L$, $\forall x,y\ [(\psi_\gamma(x)\wedge \psi_\delta(y) )\rightarrow x<_L y]$ holds in $L_{\omega_1}$. 
Since these statements hold in $L_{\hat{\tau}^\alpha}$ as well, we have $\hat{\tau}^\alpha\geq \tau^L$. 
Finally, $L_{\omega_1^L}$ and $L_{\omega_1}$ have different $\Sigma_2$-theories, since there is an uncountable cardinal in $L_{\omega_1}$, but not in $L_{\omega_1^L}$. 
Therefore $L_{(\tau^\alpha)^L}$ and $L_{\hat{\tau}^\alpha}$ have different $\Sigma_2$-theories. 
Hence $\hat{\tau}^\alpha>(\tau^\alpha)^L$.

\begin{proposition} 
If $\omega_1^L=\omega_1$, then $\tau$ equals the (strict) supremum of $L$-levels of $\Pi^1_2$-singletons. 
\end{proposition} 
\begin{proof} 
Note that $\Pi^1_2$ formulas are equivalent to $\Pi_1$-formulas over $L_{\omega_1}$. 
By the definition of $\tau$, $\tau$ is a strict upper bound for these $L$-levels. 

For the other direction, note that the $\Pi_1$-definable ordinals are cofinal in $\tau$ by Lemma 4.1 in \cite{decisiontimes}. 
It thus suffices to show that for any $\Pi_1$-definable ordinal $\alpha$, there exists a $\Pi^1_2$-singleton $x\notin L_\alpha$. 
Suppose that $\alpha$ is $\Pi_1$-definable by the formula $\varphi(x)$. 
Using $\omega_1^L=\omega_1$, let $\gamma$ be least such that $\alpha$ is countable in $L_\gamma$ and $L_{\gamma+1}\setminus L_\gamma$ contains reals. 
Let $x$ be the $L$-least code for $L_\gamma$. 
Note that $x\in L_{\gamma+1}$ by acceptability of the $L$-hierarchy \cite[Theorem 1]{boolos1969degrees}. 
Suppose that $x$ is $\Sigma_k$-definable with parameters over $L_\gamma$. 
Let $\pi:L_\gamma\rightarrow \omega$ be the unique isomorphism from $(L_\gamma,\in)$ to $(\omega,x)$ and $\pi(\alpha)=n$. 
Then $x$ is $\Pi_1$-definable by the conjunction of the statements: 
\begin{enumerate-(a)} 
\item 
\label{cond a} 
$x'$ codes an $L$-level $L_{\gamma'}$  via the pairing function. 
\item 
\label{cond b} 
From the viewpoint of $x'$, $n$ is an ordinal and 
for the corresponding ordinal $\alpha'$ in $V$ 
and any ordinal $\beta$ with $\alpha'\cong \beta$, $\varphi(\beta)$ holds. 
\item 
\label{cond c} 
$x'$ is $\Sigma_k$-definable with parameters over $L_{\gamma'}$. 
\item 
\label{cond d} 
$\gamma'$ is the least level such that $\alpha'$ is countable in $L_{\gamma'}$ and there exists a new real $\Sigma_k$-definable over $L_{\gamma'}$. 
\item 
\label{cond e} 
$x'$ is the $L$-least code for $\alpha'$ in $L_{\gamma'+1}$.\footnote{This is expressible by a $\Sigma_{k+1}$-formula over $L_{\gamma'}$, since the $\Sigma_k$-definable reals appear before the $\Sigma_{k'}$-definable ones for all $k'>k$ in the canonical wellorder of $L$. }  
\end{enumerate-(a)}
Note that \ref{cond c}-\ref{cond e}  are expressible by first-order formulas over $L_{\gamma'}$. 
\ref{cond a} and \ref{cond b} imply $\alpha'=\alpha$, 
\ref{cond a}-\ref{cond d} imply $\gamma'=\gamma$ and 
\ref{cond a}-\ref{cond e} imply $x'=x$.
\end{proof}

\section{Lengths of ranks} 
\label{section lengths of ranks} 

\subsection{The upper bound}
\label{section upper bound} 

The next lemma shows that inner models are correct about countable ranks of wellfounded $\Sigma^1_2$ relations. 
If $R$ is a wellfounded relation on a class $A$, let $\rank_R(x)$ or simply $\rank(x)$ denote the rank in $R$ of some $x\in A$. 

\begin{lemma} 
\label{correct ranks in inner models} 
Suppose that $M$ is an admissible set with $M\prec_{\Sigma^1_2}V$ and $\alpha$ is countable in $M$. 
Let $R$ be a wellfounded $\Sigma^1_2$ relation. 
\begin{enumerate-(1)} 
\item 
\label{correct ranks in inner models 1} 
``$\rank(x)\geq\alpha$'' is $\Sigma^1_2$ in any code for $\alpha$.  
\item 
\label{correct ranks in inner models 2} 
``$\rank(x)=\alpha$'' and ``$\exists x \rank(x)=\alpha$'' are absolute between $M$ and $V$. 
\end{enumerate-(1)} 
This relativises to reals. 
\end{lemma} 
\begin{proof} 
\ref{correct ranks in inner models 1}: 
Fix a countable tree $T\in M$ of rank $\alpha$, for instance the tree $[\alpha]^{<\omega}$ of finite strictly decreasing sequences below $\alpha$, ordered by inclusion. 
Then $\rank(x)\geq\alpha$ is equivalent to: 
\begin{quote} 
 \emph{``There exists $\vec{x}=\langle x_t\mid t\in T\rangle$ with $x_\emptyset=x$ and 
$(x_t,x_s)\in R$ 
for all $s\subsetneq t$ in $T$''}. 
\end{quote} 
This is $\Sigma^1_2$ in any code for $T$. 
Such a code can be chosen to be computable in any code for $\alpha$. 

\ref{correct ranks in inner models 2}: 
To see that ``$\rank(x)=\alpha$'' is absolute, fix $y\in M$ with $\rank^M(y)=\alpha$. 
We have $\rank(y)^V\geq\alpha$ by Shoenfield absoluteness. 
Towards a contradiction, suppose $\rank(y)^V>\alpha$. 
Then ``$\exists x<_R y\ \rank(x)\geq\alpha$'' holds in $V$ and hence in $M$ by \ref{correct ranks in inner models 1}. 
But $\rank^M(y)=\alpha$. 

Finally, suppose that $\exists x \rank(x)=\alpha$ holds in $V$. 
It suffices to show $\exists x \rank(x)\geq\alpha$ in $M$. 
This holds by \ref{correct ranks in inner models 1}. 
\end{proof} 

The next result is an effective version of the Kunen-Martin theorem for $\Sigma^1_2$ relations of countable rank.  

\begin{proposition} 
\label{Kunen-Martin for Sigma12} 
The rank of any wellfounded  $\Sigma^1_2$ relation is either uncountable or ${<}\tau$. 
\end{proposition} 
\begin{proof} 
Let $\alpha$ denote that rank of a wellfounded $\Sigma^1_2$ relation $R$. 
Suppose that $\alpha$ is countable. 
Let $\psi(\beta)$ denote the following $\Pi_1$-formula in $\beta$: 
\begin{quote} 
``\emph{For all countable $\gamma$ with $L_\gamma\models \ZFC^-$, there exists a strict order preserving function $f\colon L_\gamma \rightarrow \beta$ with respect to $R^{L_\gamma}$ and $\in$.}'' 
\end{quote}  
We have $R^{L_\beta}\subseteq R$ for any such $L_\gamma$, since $R$ is $\Sigma^1_2$ and $L_\gamma$ is $\Sigma^1_1$-correct in $V$. 
Therefore, $R^{L_\gamma}$ is wellfounded with rank at most $\alpha$, so $\psi(\alpha)$ holds in $V$. 

Note that $\psi(\alpha)$ is $\Pi^1_2$ in any code for $\alpha$, so it is $\Pi_1$ in $\alpha$ in any model where $\alpha$ is countable. 
We thus need to collapse $\alpha$. 
Let $G$ be a $\Col(\omega,\alpha)$-generic filter over $V$ and note that $R^{V[G]}$ is wellfounded in $V[G]$ by Shoenfield absoluteness, since this is a $\Pi^1_2$ property. 
Denote the rank of $R$ by $\rank_R$. 
We have $\rank_R^V=\rank_R^{V[G]}=\rank_R^{L[G]}$ by Lemma \ref{correct ranks in inner models} for $L$ and $L[G]$ in $V[G]$ 
This shows $\psi(\alpha)$ in $L[G]$ and thus $\Vdash_{\Col(\omega,\alpha)} \psi(\alpha)$ in $L$. 
Note that the latter statement is $\Pi_1$ in $\alpha$ as well.  
Thus the $\Sigma_2$-statement $\exists \beta<\omega_1 \Vdash_{\Col(\omega,\beta)} \psi(\beta)$ holds in $L_{\omega_1}$. 
Pick some $\beta<\tau$ with $\Vdash_{\Col(\omega,\beta)}^{L_{\omega_1}} \psi(\beta)$ by the definition of $\tau$, so  $\psi(\beta)$ holds in $L[G]$. 
Then $\psi(\beta)$ holds in $V$ by Shoenfield absoluteness. 
Now pick a countable $\gamma$ such that $R\cap L_\gamma$ has rank $\alpha$ and $L_\gamma\prec L_{\omega_1}$. 
Since $\psi(\beta)$ holds in $V$, we have $\rank_{R\cap L_\gamma}=\alpha\leq\beta<\tau$ as required. 
\end{proof}

In particular, the length of any $\Sigma^1_2$ rank of countable length is strictly less than $\tau$.

\subsection{The lower bound}
\label{section - the lower bound} 

Suppose that $\alpha$ is a countable ordinal. 
Recall that an ordinal $\beta$ is an \emph{$\alpha$-index} if $\beta>\alpha$ and some $\Sigma_1$ fact in $L_{\omega_1}$ with parameters ${\leq}\alpha$ first becomes true in $L_\beta$. 
Thus $\sigma_\alpha$ is the supremum of $\alpha$-indices. $\sigma_{\alpha}$ is an admissible limit of admissible ordinals.

We shall construct a $\Pi^1_1$ rank of length at least $\nu$ for unboundedly many $\nu<\tau$. 
By \cite[Lemma 4.1]{decisiontimes}, there are unboundedly many $\nu<\tau$ such that $\nu$ is $\Pi_1$-definable over $L_{\omega_1}$. 
Take such an ordinal $\nu$ and fix a $\Pi_1$-formula $\varphi(x)$ that defines $\nu$ over $L_{\omega_1}$. 
Similar to \cite[Theorem 4.5]{decisiontimes}, we shall define a $\Pi^1_1$ subset $A$ of $\WO$ and a $\Pi^1_1$-rank of length $\sigma_\nu$ on $A$. 
Recall the notation $\alpha_x=\otp(x)$. 
The first approximation to $A$ is 
$$A^\nu:= \{ x\in \WO \mid \alpha_x>\nu\wedge \alpha_x \text{ is a $\nu$-index and $L_{\alpha_x} \models$ ``$\nu$ is least with $\varphi(\nu)$''}\}.$$ 
For any $\alpha>\nu$, $\varphi(\nu)$ holds in $L_\alpha$ by $\Pi_1$ downwards absoluteness. 
Moreover, by the definition of $\sigma_\nu$ there is some $\gamma<\sigma_\nu$ such that witnesses for the failure of $\varphi(\mu)$ for all $\mu<\nu$ appear before $\gamma$. 
Therefore, $x\in A^\nu$ for any $\nu$-index $\alpha_x>\gamma$. 
The $\nu$-indices $\alpha_x$ have supremum $\sigma_\nu$ and we will show that they have order type $\sigma_\nu$. 
The rank will simply be the restriction to $A$ of the rank of $\WO$, thus its length will be at least $\sigma_\nu$. 

Since we want to define $A$ by a $\Pi^1_1$-formula, we cannot use $\nu$ in the definition of $A$. 
We therefore vary $\nu$ in the definition of $A^\nu$: 
for any countable ordinal $\mu$, let 
$$A^\mu:= \{ x\in \WO \mid \alpha_x>\mu\wedge \alpha_x \text{ is a $\mu$-index and $L_{\alpha_x} \models$ ``$\mu$ is least with $\varphi(\mu)$''}\}.$$ 
Finally, let 
$$ A=  \bigcup_{\mu<\omega_1} A^\mu. $$  
This is a $\Pi^1_1$ set, since $x\in A$ if and only if $x\in \WO$, there is some $\mu<\alpha_x$ with $L_{\alpha_x}\models \varphi(\mu)$, $\mu$ is the least such ordinal, and $\alpha_x$ is a $\mu$-index. 
For the next proof, we shall take this as our formal definition of $A$. 
Note that $A$ is an intersection of $\WO$ with a $\Delta^1_1$ set, since the second part of the definition is first-order over $L_{\alpha_x}$. 
Moreover, it is easy to see that $A^\mu=\emptyset$ for all $\mu>\nu$ and $A^\mu$ is bounded by $\gamma$ for all $\mu<\nu$. 
Hence $A$ is a subset of $\WO_{<\nu}$. 
If $x\in A$, we shall write $\mu_x$ for the unique $\mu$ with $L_{\alpha_x} \models$ ``$\mu$ is least with $\varphi(\mu)$''. 

\begin{theorem} 
\label{lower bound for lengths of Pi11 ranks} 
The lengths of $\Pi^1_1$ ranks of countable length are unbounded in $\tau$. 
\end{theorem} 
\begin{proof} 
Suppose that $\nu$ is $\Pi_1$-definable over $L_{\omega_1}$. 
Define $A$ as above. 
Let $\LO$ denote the set of linear orders on $\omega$. 
For $x,y\in \LO$, let $x\leq_\wo y$ if $\otp(x)\leq\otp(y)$ or $y$ is illfounded, and $x<_\wo y$ if $\otp(x)<\otp(y)$ or $y$ is illfounded. 
This defines a $\Pi^1_1$ rank on $\WO$. 

\begin{claim*} 
${\leq_\wo}{\upharpoonright}A$ induces a $\Pi^1_1$ rank on $A$. 
\end{claim*} 
\begin{proof} 
Define the following $\Pi^1_1$ relations $\sqsubseteq$ and $\sqsubset$. 
Recall that $A$ is the intersection of $\WO$ with a $\Delta^1_1$ set $B$. 
Let $x \sqsubset y$ if $x\in A$ and one of the following holds: 
\begin{itemize} 
\item 
$y\notin \LO$, 
\item 
$y\in \LO$ codes a linear order $\LL$ that is not embeddable into $\otp(x)$, or 
\item 
$y\notin B$. 
\end{itemize} 
Moreover, let $x \sqsubseteq y$ if $x\sqsubset y$ or \begin{itemize} 
\item 
$y\in \LO$ codes a linear order $\LL$ that is not embeddable into a strict initial segment of $\otp(x)$. 
\end{itemize} 
These relations agree with $<_\wo$ and $\leq_\wo$ on $A$ and satisfy overspill. 
They thus define a $\Pi^1_1$ rank on $A$. 
\end{proof} 

\begin{claim*} 
$|{<_\wo}{\upharpoonright}A| \leq\sigma_\nu$.  
\end{claim*} 
\begin{proof} 
It suffices to show that  $\alpha_x<\sigma_\nu$ for all $x\in A$. 
We have $\alpha_x<\sigma_{\mu_x}$, since $\alpha_x$ is a $\mu_x$-index by the definition of $A$. 
Moreover, $\mu_x\leq \nu$ by the definition of $\mu_x$. 
Since the function $\alpha\mapsto\sigma_\alpha$ is monotone, we have $\alpha_x<\sigma_{\mu_x}\leq \sigma_\nu$. 
\end{proof} 

\begin{claim*} 
$|{<_\wo}{\upharpoonright}A|\geq \sigma_\nu$. 
\end{claim*} 
\begin{proof} 
We shall construct a strictly increasing sequence $\vec{\xi}=\langle \xi_\alpha \mid \alpha<\sigma_\nu\rangle$ by a $\Delta_0$-recursion. 
Let $\xi_0=\nu$. 
Given $\langle \xi_\alpha \mid \alpha<\gamma\rangle$ for $1\leq\gamma<\sigma_\nu$, let $\xi_\gamma$ be the least $\nu$-index $\xi>\sup_{\alpha<\gamma} \xi_\alpha$, i.e. a new $\Sigma_1$-fact with parameters ${\leq}\nu$ becomes true in $L_{\xi_\gamma}$. 
Since $\sigma_\nu$ is admissible, $\vec{\xi}$ is a well-defined sequence of length $\sigma_\nu$. 
For all $\alpha<\sigma_\nu$, $A_\nu$ contains all codes for $\xi_\alpha$, since $\xi_\alpha\geq \nu$ is a $\nu$-index. 
Thus $A$ contains a strictly increasing sequence of length $\sigma_\nu$. 
\end{proof} 
Thus the rank's length is precisely $\sigma_\nu$. 
\end{proof} 

The next result strengthens the previous result by showing that $A$ admits no $\Pi^1_1$ rank shorter than $\sigma_\nu$. 
Let $\otp\colon\WO\rightarrow \omega_1$ denote the function mapping a real in $\WO$ to its order type. 
Thus $\otp[B]$ denotes the pointwise image of a subset $B$ of $\WO$. 

\begin{proposition} 
\label{no short ranks} 
Suppose that $\nu$ is a countable ordinal and $B$ is a $\Pi^1_1$ set with $\otp[B]\geq \sigma_\nu$. 
Then $B$ admits no $\Pi^1_1$ rank of length ${<}\sigma_\nu$. 
\end{proposition} 
\begin{proof} 
Towards a contradiction, suppose that $\sqsubseteq$ and $\sqsubset$ define a rank on $B$ of length $\delta<\sigma_\nu$. 
Let $\psi(\vec{\alpha},\vec{x},\gamma)$ denote the conjunction of the statements: 
\begin{itemize} 
\item  
``\emph{$\vec{\alpha}=\langle \alpha_i\mid i<\gamma\rangle$ is a strictly increasing sequence of ordinals}.'' 
\item 
``\emph{$\vec{x}=\langle x_i\mid i<\gamma\rangle$ with $x_i\in B$ is strictly increasing with respect to $\sqsubset$}.'' 
\item 
``\emph{$\otp(x_i)=\alpha_i$ for all $i<\gamma$}.'' 
\end{itemize} 
$\psi$ is a $\Sigma_1$-formula, since ``$x_i\in B$'' and ``$x_i\sqsubset x_j$'' are $\Pi^1_1$ and thus $\Sigma_1$. 
The statement $\exists \vec{x}\ \exists \vec{\alpha}\ \psi(\vec{x},\vec{\alpha},\delta)$ holds in $V$ and hence also in $L_{\sigma_\delta}$. 
Let $\vec{x}, \vec{\alpha} \in L_{\sigma_\delta}$ witness this. 
Now let $g\in V$ by $\Col(\omega,\delta)$-generic over $L_{\sigma_\nu}$. 
Since $\gamma\leq\nu$, we have $\sigma_\gamma\leq \sigma_\nu$. 
Pick $u\in L_{\sigma_\nu}[g]$ with $\otp(u)=\delta$ and let $\vec{v}=\langle v_i \mid i\in\omega \rangle$ denote the enumeration of $\vec{x}$ along $u$, i.e.,  
$$\forall i,j\ ((i,j)\in u \Rightarrow v_i\sqsubset v_j)$$ 
holds. 
The latter is $\Pi^1_1$ in $u$ and $\vec{y}$. 
Since $L_{\sigma_\nu}[g]$ is a strictly increasing union of admissible sets, $L_{\sigma_\nu}[g]\prec_{\Sigma^1_1}V$. 
Hence the previous statement holds in $V$. 
It follows that in $V$, $\vec{u}$ is a strictly increasing sequence with respect to $\sqsubset$ of length $\delta$. 
Therefore, $\vec{u}$ is unbounded in the rank. 
Thus  $x\in B$ $\Longleftrightarrow$ $\exists i\in\omega\ x \sqsubseteq u_i$ is a $\Sigma^1_1$ definition of $B$ in $\vec{u}$. 
Then $B$ is bounded by $\omega_1^{\ck,\vec{u}}$ by the effective Kunen-Martin theorem (see \cite[Lemma 4.4]{decisiontimes}). 
Since $L_{\sigma_\nu}[g]$ is a strictly increasing union of admissible sets, we have $\omega_1^{\ck,\vec{u}}<\sigma_\nu$. 
But $B$ is unbounded in $\sigma_\nu$ by definition. 
\end{proof} 

Next is an analogous result for $\Sigma^1_2$-ranks. 

\begin{proposition} 
\label{Sigma12 long ranks} 
For any $\Pi_1$-definable $\nu$, there exists a $ \Sigma^1_2$ set that admits a countable $\Sigma^1_2$-rank of length $\sigma_\nu$, but none shorter than $\sigma_\nu$. 
\end{proposition} 
\begin{proof} 
Let $A$ denote the set defined relative to $\nu$ in the beginning of Section \ref{section - the lower bound} and suppose that $\sqsubseteq$ comes from the standard rank on $\WO$. 
Then 
$$\WO_{{<}\sigma_\nu}= \{ x\in \WO \mid \exists y\in A\ x\sqsubseteq y \} $$ 
is a $\Sigma^1_2$ set. 
It is easy to see that the restriction of the standard rank on $\WO$ to $\WO_{{<}\sigma_\nu}$ is a $\Sigma^1_2$-rank on $\WO_{{<}\sigma_\nu}$ of length $\sigma_\nu$. 
It remains to show that $\WO_{{<}\sigma_\nu}$ does not support a $\Sigma^1_2$-rank of length $\alpha<\sigma_\nu$. 
Aiming for a contradiction, suppose that there exists such a rank with relation $\unlhd$. 
Then $x\in \WO_{{\geq}\sigma_\nu}$ if and only if 
\begin{quote} 
``\emph{$x\in \WO$ and there exists a sequence $\langle x_i \mid i<\alpha\rangle$ of reals, strictly increasing in the rank, with $\forall i<\alpha\ x_i \unlhd x$}.'' 
\end{quote} 
This statement $\psi(x)$ is $\Sigma^1_2$ in any real coding $\alpha$. 
For any $\Col(\omega,\alpha)$-generic filter $G$ over $V$, 
$\exists x\ \psi(x)$ holds in $V[G]$ and $L[G]$ by Shoenfield absoluteness. 
Since $L_{\sigma_\nu}[G]\prec_1L[G]$, this holds in $L_{\sigma_\nu}[G]$. 
Then $\exists x\ \psi(x)$ also holds in $L_{\sigma_\nu}[g]$ for any $\Col(\omega,\alpha)$-generic filter over $L_{\sigma_\nu}$ in $V$ by homogeneity of $\Col(\omega,\alpha)$. 
Note that every $\Sigma^1_2$ statement about a real in $L_{\sigma_\nu}[g]$ holds in $V$, since $L_{\sigma_\nu}[g]$ is $\Sigma^1_1$-correct in $V$. 
Since $\exists x\ \psi(x)$ holds in $L_{\sigma_\nu}[g]$, it follows that $\WO_{{\geq}\sigma_\nu}$ contains a real in $L_{\sigma_\nu}[g]$. 
But any real in $\WO\cap L_{\sigma_\nu}[g]$ has order type less than $\sigma_\nu$, since $L_{\sigma_\nu}[g]$ is admissible. 
\end{proof}

\subsection{Beyond $\tau$}
\label{section - beyond tau} 

We now prove the results described in Figure \ref{figure length of pwos}. 

\begin{remark} 
The (strict) supremum of lengths of countable strict $\Sigma^1_1$ prewellorders is $\omega_1^\ck$ by the effective version of the Kunen-Martin theorem. (This is immediate from the proof of \cite[Theorem 31.1]{kechris2012classical}.) 
It follows that the  suprema of lengths of $\Delta^1_1$ prewellorders on $\Delta^1_1$ sets and $\Pi^1_1$ prewellorders on $2^\omega$ are $\omega_1^\ck$. 
\end{remark} 

\begin{remark} 
Every $\Pi^1_1$-rank on a $\Pi^1_1$ set $A$ induces a $\Sigma^1_1$ prewellorder on $2^\omega$ by letting $x\leq y$ if $x\sqsubseteq y$ or $x,y\notin A$, where $\sqsubseteq$ comes from the rank. 
It also induces a $\Pi^1_1$ prewellorder on $A$. 
Hence both suprema are at least $\tau$. 
It follows that the suprema of all classes in the second column of Figure \ref{figure length of pwos} are at least $\tau$. 
They all equal $\tau$, since the length of strict $\Sigma^1_2$ prewellorders is bounded by $\tau$ by Proposition \ref{Kunen-Martin for Sigma12} and this class of prewellorders is more general than the others. 
\end{remark} 
 
\begin{proposition} 
\label{long Pi12 prewellorder in ZFC} 
There exists a $\Pi^1_2$ prewellorder $\leq$ on a $\Pi^1_2$ set $A$ of length $\tau$. 
In fact, both $\leq$ and $<$ can be chosen to be $\Pi^1_2$. 
\end{proposition} 
\begin{proof} 
Let $A$ denote the set of triples $(x,y,\varphi)$, where $x,y\in \WO$ and $\varphi(z)$ is a $\Pi_1$-formula with the properties: 
\begin{enumerate-(a)} 
\item 
\label{long Pi12 prewellorder 1} 
$x$ codes $\alpha$. 
\item 
\label{long Pi12 prewellorder 1a} 
$y\in L_\alpha$. 
\item 
\label{long Pi12 prewellorder 2} 
Given $\alpha$, $\beta$ is least with $L_\alpha \models \varphi(\beta)$. 
\item 
\label{long Pi12 prewellorder 3} 
Given $\beta$, $\alpha$ is least with \ref{long Pi12 prewellorder 2}. 
\item 
\label{long Pi12 prewellorder 4} 
$L_{\omega_1} \models \varphi(\beta)$, equivalently $L_{\omega_1} \models \forall \gamma>\beta\ L_\gamma\models \varphi(\beta)$.\footnote{More precisely, read this as a $\Pi_1$-definition in a natural number coding $\varphi$.}  
\end{enumerate-(a)} 
$A$ is $\Pi^1_2$, since \ref{long Pi12 prewellorder 4} is $\Pi^1_2$ and the remaining conditions are arithmetical. 
Define a prewellorder $\leq$ on $A$ by letting $(x,y,\varphi)\leq (u,v,\psi)$ if $\otp(y)\leq\otp(v)$. $<$ is defined similary. 
Both $\leq$ and $<$ are relative $\Sigma^1_1$ and $\Pi^1_1$ relations on $A$. 

We claim that the order type of these prewellorders is precisely $\tau$. 
To see that it is at most $\tau$, we show that $A\subseteq L_\tau$. 
For any $\Pi_1$-formula $\varphi$, there is at most one $x$ such that $(x,y,\varphi)\in A$ for some $y$. 
Moreover, the existence of some $x,y$ with $(x,y,\varphi)\in A$ is $\Sigma_2$ over $L_{\omega_1}$, so any such $x$ is in $L_\tau$. 
It follows that $y\in L_\tau$. 

To see that it is at least $\tau$, recall that $\tau$ equals the supremum of $\Pi_1$-definable ordinals over $L_{\omega_1}$. 
If $\beta$ is $\Pi_1$-definable by $\varphi(z)$, then \ref{long Pi12 prewellorder 2}-\ref{long Pi12 prewellorder 4} hold for the least $\alpha$ with \ref{long Pi12 prewellorder 2}. 
Since $y$ may code any ordinal below $\alpha$, $\otp(\leq)$ is at least $\beta$. 
\end{proof} 

Using the previous construction, one can easily find countable (strict) $\Pi^1_2$ prewellorders in $\Pi^1_2$ sets of length $\tau+\tau$, $\tau\cdot\tau$ and more. 
However, we do not know if it is provable in $\ZFC$ that the supremum of $\Pi^1_2$ prewellorders on $2^\omega$ is larger than $\tau$. 
In particular, this is open in $L$. 

\begin{remark} 
The supremum of countable strict $\Pi^1_2$ prewellorders on $2^\omega$ is highly variable in forcing extensions of $L$.  
For instance, for any ordinal $\alpha\geq 1$, there is a generic extension 
of $L$ where $\omega_1$ is preserved and 
there is a $\Sigma^1_2$ prewellorder of length precisely $\alpha$. 
To see this, one realises $\alpha$ as the order type of the constructibility degrees of reals in an iteration of Sacks forcing (see \cite[Lemma 6]{miller1983mapping}). 
\end{remark} 

We now utilise large cardinals to say more about this supremum. 
Assume that $0^\#$ exists \cite[Section 18]{Je03}. 
Note that $0^\#$ exists if and only if for some limit ordinal $\lambda$, $(L_\lambda,\in)$ has an uncountable set of order indiscernibles \cite[Corollary 18.18]{Je03} or equivalently, there is a nontrivial elementary embedding $j\colon L\rightarrow L$ \cite[Theorem 18.20]{Je03}.  
Let $\iota_0$ denote the first Silver indiscernible. 
This equals the least critical point of any elementary embedding $j\colon L\rightarrow L$. 

%
%
%

\begin{proposition} 
Assuming $0^\#$ exists, there exist (strict) $\Pi^1_2$ prewellorders on countable $\Pi^1_2$ sets of length strictly above $\iota_0$. 
\end{proposition} 
\begin{proof} 
The existence of $0^\#$ is equivalent to the existence of an iterable structure $(L_\alpha,\in,U)$ such that $(L_\alpha,\in)$ is a model of $\ZFC^-$ and $U$ is amenable to $L_\alpha$, i.e. $x\cap U\in L_\alpha$ for all $x\in L_\alpha$ \cite[Definition 10.37, Theorem 10.39 \& Corollary 10.44]{Sch14}. 
Iterable means that all countable (and thus all) iterated ultrapowers are wellfounded. 
This structure is unique if we additionally assume that it is the $\Sigma_1$-Skolem hull of the empty set in itself \cite[Corollary 10.36]{Sch14}. 
Denote this structure by $M_0^\#=(L_\alpha,\in,U)$. 
The first Silver indiscernible $\iota_0$ equals the critical point of $U$ \cite[Corollary 10.44]{Sch14}. 

There exists a $\Sigma_1$-definable surjection $f\colon\omega\rightarrow L_\alpha$ over $M_0^\#$, since $M_0^\#$ is the $\Sigma_1$-Skolem hull of the empty set. 
The pair of $\in_f:=\{(m,n)\mid f(m) \in f(n) \}$ and $U_f:= f^{-1}[U]$ 
is a \emph{canonical code} for $M_0^\#$ that can be defined by a $\Pi^1_2$-formula $\varphi(x,y)$. 

Using this, we define a $\Pi^1_2$ prewellorder as follows. 
The underlying set consists of all triples $(x,y,n)\in 2^\omega\times 2^\omega\times \omega$ such that $\varphi(x,y)$ holds and $(\omega,x)\models ``n$ is an ordinal''. 
Recall that $x$ and $y$ are unique. 
Let $(x,y,m)\leq (x,y,n)$ if $(\omega,x)\models ``m\leq n"$. 
$<$ is defined similarly. 
Their lengths equal the height of $M_0^\#$ and are thus strictly above $\iota_0$. 
\end{proof} 

It remains to consider countable $\Pi^1_2$ prewellorders on $2^\omega$. 

\begin{proposition} 
Assuming $0^\#$ exists, there exists a strict $\Pi^1_2$ prewellorder on $2^\omega$ of countable length strictly above $\iota_0$. 
\end{proposition} 
\begin{proof} 
It suffices to define such a prewellorder with a $\Delta^1_2$ domain, since we can extend this by adding the complement as a single equivalence class. 
In fact, the domain will be $\Pi^1_1$. 
It consists of all reals coding structures of the form $\MM=(M,\in,U,\alpha)$, where 
$(M,\in)$ is a wellfounded model of $\ZFC^-$, $U$ is amenable to $M$,  $M=L_\gamma$ for some $\gamma$, $\gamma$ is minimal with these properties (i.e. for all $\beta<\gamma$, $(L_\beta,\in,U)$ does not satisfy all the previous conditions), $(L_\gamma,\in,U)$ is sound (i.e., the Skolem hull of the empty set), 
and $\alpha\in \Ord^M$. 
We then write $\alpha_\MM=\alpha$. 
Note that the domain is nonempty, since $M_0^\#$ is such a structure. 

An \emph{iteration} of $\MM$ is a sequence of iterated ultrapowers\footnote{I.e., standard internal ultrapowers together with a predicate for an ultrafilter using amenability.}  of $\MM$ where each model is wellfounded. 
If $\MM_\gamma=(M_\gamma,\in,U_\gamma,\delta)$ is an iterate of $\MM$, write $\MM_\gamma^*=(M_\gamma,\in,U_\gamma)$. 
A \emph{coiteration} of $\MM$ and $\NN$ is a pair of iterations of these structures as defined in \cite[Section 4.4]{zeman2011inner}. 
A coiteration is called \emph{successful} if it has final models $\MM_\gamma$ and $\NN_\gamma$ with $\MM_\gamma^* \unlhd \NN_\gamma^*$\footnote{I.e., $\MM_\gamma^*$ is a (not necessarily proper) initial segment of $\NN_\gamma^*$.} or $\NN_\gamma^*\unlhd \MM_\gamma^*$. 

If $\MM$ is iterable but $\NN$ is not, then there is no successful coiteration of $\MM$ and $\NN$. 
To see this, suppose that the final models are $\MM_\gamma$ and $\NN_\gamma$. 
If $\NN_\gamma^*\unlhd \MM_\gamma^*$, then $\NN$ would be iterable. 
If $\MM_\gamma^*\lhd \NN_\gamma^*$, then $\NN_\gamma^*$ and thus also $M_0^\#$ has a proper initial segment with the same theory. 
This contradicts the minimality of $M_0^\#$. 

If $\MM$ and $\NN$ are both iterable, we claim that none is moved in the coiteration. 
Thus $\MM^*=\NN^*=M_0^\#$. 
To see this, note that at most one is moved by a general fact. 
Towards a contradiction, suppose the final models are $\MM_\gamma\unlhd\NN_\gamma$ and $\MM$ is moved. 
If $\MM_\gamma=\NN_\gamma=\NN$, then $\NN$ would not be sound, as $\MM$ is moved. 
If $\MM_\gamma\lhd\NN_\gamma=\NN$, then $\NN$ would not be minimal. 

Let $\MM< \NN$ if $\MM$ is iterable, i.e. all its countable iterated ultrapowers are wellfounded, and for every successful countable coiteration of $\MM$ and $\NN$ with final models $\MM_\gamma$, $\NN_\gamma$ and iteration maps $\pi\colon M\rightarrow M_\gamma$, $\nu\colon N\rightarrow N_\gamma$, we have 
$\MM_\gamma^*=\NN_\gamma^*$ and $\pi(\alpha_\MM)\leq \nu(\alpha_\NN)$. 
Then $\MM<\NN$ holds if and only if: 
\begin{itemize} 
\item 
$\MM^*=\NN^*=M_0^\#$ and $\alpha_\MM<\alpha_\NN$, or 
\item 
$\MM^*=M_0^\#$ but $\NN$ is not iterable.\footnote{We work with the above indirect definition, since a case distinction according to iterability is too complex for a $\Pi^1_2$ definition.} 
\end{itemize} 
Therefore, the length of $<$ equals $\Ord^{M_0^\#}+1$ and is thus strictly above $\iota_0$. 
\end{proof} 



\subsection{Sets with countable ranks} 
\label{section - sets w ctbl ranks} 

We aim to characterise sets that support a countable rank. 
We shall use this to calculate the supremum of lengths of ranks on co-countable $\Sigma^1_2$ sets in the next section. 
For $\Pi^1_1$ sets, this is easy: 

\begin{proposition} 
\label{Borel Pi11 sets and countable ranks} 
A $\Pi^1_1$ set admits a countable $\Pi^1_1$-rank if and only if it is Borel. 
In fact, all $\Pi^1_1$-ranks on $\Pi^1_1$ Borel sets are countable. 
\end{proposition} 
\begin{proof} 
A $\Pi^1_1$ set with a countable $\Pi^1_1$-rank is a countable union of Borel sets and thus itself Borel. 
Conversely, suppose that $A$ is a $\Pi^1_1$ Borel set and fix any $\Pi^1_1$-rank on $A$ with relations $\sqsubseteq$ and $\sqsubset$. 
Then the relation defined by $x,y\in A \wedge x\sqsubset y$ is a strict ${\bf\Sigma}^1_1$ prewellorder on $A$. 
By the Kunen-Martin theorem \cite[Theorem 31.1]{kechris2012classical}, the rank's length is countable. 
\end{proof} 

\begin{remark} 
Beyond $\Pi^1_1$, sets can support both countable and uncountable ranks. 
To see this, note that there exist $\Sigma^1_1$ prewellorders on $2^\omega$ of length $\omega_1$. 
If  $\Gamma$ is a class strictly above $\Pi^1_1$ with the rank property that is closed under continuous preimages, then 
any $\Gamma$ set that contains a computable copy of $2^\omega$ admits some uncountable $\Gamma$-rank. 
\end{remark} 

The forward implication of Lemma \ref{Borel Pi11 sets and countable ranks} fails just beyond $\Pi^1_1$. 
To see this, suppose that $\Gamma$ contains both $\Pi^1_1$ and $\Sigma^1_1$. 
Let $\Delta$ denote the class of sets that are both $\Gamma$ and $\bar{\Gamma}$. 
There exist $\Delta$ sets which are not Borel and any such set admits a $\Gamma$ rank of length $1$. 
We shall show that the reverse implication of Lemma \ref{Borel Pi11 sets and countable ranks} fails for $\Sigma^1_2$ sets. 
For instance, the complement of the singleton $0^\#$ does not admit a countable $\Sigma^1_2$ rank by Theorem \ref{characterisation countable ranks}. 


The failure of both implications suggests to study the relationship between the  properties for any $\Sigma^1_2$ Borel set $A$: 
\begin{itemize} 
\item 
$A$ admits a countable $\Sigma^1_2$ rank. 
\item 
$A$ has a \borel code in $L_\tau$. 
\end{itemize} 

We shall show that these properties are equivalent for co-countable $\Sigma^1_2$ in Theorem \ref{characterisation countable ranks} below, assuming $\Sigma^1_3$ Cohen absoluteness. 
(Note that this holds trivially for countable $\Sigma^1_2$ sets.) 

In the following, if $A$ is a definable set or relation, we will assume that $A$ is given by a fixed definition. 
It then makes sense to talk about the version of $A$ in a generic extension $V[G]$ of $V$. 
For instance, we write $A^{V[G]}$ for the set or relation with the same definition in $V[G]$.

\begin{lemma} 
\label{countable Pi12 in inner models} 
Suppose that $M$ is an admissible set 
with $M\prec_{\Sigma^1_2} V$ and $\gamma$ is countable in $M$. 
Let $A$ be a $\Pi^1_2$ set whose complement admits a countable $\Sigma^1_2$-rank of length $\gamma$. 
Then $A$ is $\Sigma^1_2$ in some real in $M$ in any outer model $W$ where $\sqsubseteq^W$ and $\sqsubset^W$ form a rank on $A^W$ of the same length as in $V$. 
This relativises to reals. 
\end{lemma} 
\begin{proof} 
``\emph{There exists a sequence $\langle x_\alpha\mid \alpha<\gamma\rangle$ in the complement of $A$ with $\rank(x_\alpha)\geq \alpha$ for all  $\alpha<\gamma$}'' is $\Sigma^1_2$ in a real in $M$ by Lemma \ref{correct ranks in inner models} \ref{correct ranks in inner models 1}. 
Since $M\prec_{\Sigma^1_2} V$, pick such a sequence $\vec{x}\in M$ and let $y\in M$ code $\vec{x}$. 
Let $\sqsubseteq$, $\sqsubset$ denote the relations of the rank. 
Then for all reals $x$: 
$$ x\in A \Longleftrightarrow \forall \alpha<\gamma\ (x_\alpha\sqsubseteq x). $$ 
This is $\Sigma^1_2$ in $y$. 
\end{proof} 

Note that for any countable $\Sigma^1_2$-rank on a $\Sigma^1_2$ set $A$ that remains a rank in an outer model $W$, its length also remains the same in $W$, since the statement ``\emph{$\langle x_\alpha \mid \alpha<\gamma\rangle$ is cofinal in the rank}'' is $\Pi^1_2$ in a real and hence absolute to $W$ by Shoenfield absoluteness. 


\begin{definition} 
\label{def stable ranks} 
Suppose that $A$ is a $\Sigma^1_2$ set and $\sqsubseteq$, $\sqsubset$ define a $\Sigma^1_2$ rank on $A$. 
We call this rank \emph{$\PP$-stable} if in all $\PP$-generic extensions $V[G]$ of $V$, $\sqsubseteq^{V[G]}$, $\sqsubset^{V[G]}$ form a rank on $A^{V[G]}$. 
\end{definition} 

Note that we do not require that $A^V=A^{V[G]}$ holds in $\PP$-generic extensions $V[G]$ of $V$. 
However, the length of any countable $\PP$-stable $\Sigma^1_2$-rank remains the same in any $\PP$-generic extension $V[G]$ by Shoenfield absoluteness. 
Moreover, note that every $\Sigma^1_2$ set admits a $\Sigma^1_2$-rank that is $\PP$-stable for all forcings $\PP$, since the proof of the existence of $\Sigma^1_2$-ranks shows that a fixed definition of the relations works and thus, these also form a rank in generic extensions. 

We shall now study countable $\Pi^1_2$ sets. 
Note that a countable set $A$ is an element of $L_\tau$ if and only if it has a Borel code in $L_\tau$ by the Mansfield-Solovay theorem \cite[Theorem 25.23]{Je03} and Lemma \ref{Sigma1 elementarity of Ltau}.\footnote{For $L_\alpha \prec_{\Sigma^1_1} L$, it suffices to know that $\alpha$ is a limit of admissibles.}  

\begin{lemma} 
\label{stable countable rank implies constructible} 
Suppose that $A$ is a countable $\Pi^1_2$ set whose complement admits a countable 
Cohen stable $\Sigma^1_2$-rank. 
If $\omega_1^L=\omega_1$, then $A\in L_\tau$. 
If $\omega_1^L$ is countable, then $A\in L_{\omega_1^L+1}$. 
Moreover, these $L$-levels are optimal. 
Furthermore, the same result holds for arbitrary countable $\Sigma^1_2$-ranks assuming for every $\gamma$ that is countable in $V$ but uncountable in $L$, there exists a $\Col(\omega,\gamma)$-generic filter over $L$ in $V$.\footnote{It is open whether this assumption is provable in $\ZFC$.}  
For instance, this is the case if $\omega_1$ is inaccessible in $L$. 
\end{lemma} 
\begin{proof} 
We split the proof into two claims. 
\begin{claim*} 
$A\subseteq L$. 
\end{claim*} 
\begin{proof} 
We shall use that for any countable $\Sigma^1_2(x)$ set $B$, $B=B^W$ holds in all outer models $W$, 
since the statement ``\emph{every real in $B$ appears in $\vec{x}=\langle x_i \mid i\in\omega\rangle$}'' is $\Pi^1_2$ in $\vec{x}$ and $x$ and thus it is absolute to outer models by Shoenfield absoluteness. 

By Lemma \ref{countable Pi12 in inner models}, $A$ is $\Sigma^1_2$ in a real and the definition remains valid in Cohen generic extensions of $V$. 
Here we use that the rank is Cohen stable. 
By the previous remark, $A=A^{V[G]}$ for Cohen generic extensions of $V$. 
In particular, $A$ is countable in $V[G]$. 
Let $\gamma$ denote the length of the given rank. 
Let $G$ and $H$ be mutually $\Col(\omega,\gamma)$-generic filters over $V$. 
Note that  $\Col(\omega,\gamma)\times \Col(\omega,\gamma)$ is countable and hence forcing equivalent to Cohen forcing.\footnote{Forcings $\PP$ and $\QQ$ are called equivalent if their Boolean completions are isomorphic.} 
Thus $A=A^{V[G,H]}$ is countable in $V[G,H]$. 
By Lemma \ref{countable Pi12 in inner models}, $A$ is $\Sigma^1_2$ in a real in $L[G]$ and this definition of $A$ is valid in $V[G,H]$. 
Thus $A\subseteq L[G]$ by the Mansfield-Solovay theorem. 
Similarly $A\subseteq L[H]$ and thus $A\subseteq L[G]\cap L[H]=L$. 
\end{proof} 

If $\omega_1^L$ is countable, it follows that $A\in L_{\omega_1^L+1}$. 
Note that $\omega_1^L+1$ is optimal, since the countable $\Pi^1_1$ set of $L$-least codes for countable $L$-levels has $L$-rank $\omega_1^L+1$. 
It remains to show $A\in L_\tau$. 
Note that the proof of Lemma \ref{ranks on countable sets} \ref{ranks on countable sets 1} shows that $\tau$ is optimal as well. 

\begin{claim*}
$A\subseteq L_\delta$ for some $\delta<\tau$. 
\end{claim*} 
\begin{proof} 
Since the rank's length does not increase in Cohen generic extensions of $V$, the following statement $\psi(\gamma)$ holds in $V$ for some countable $\gamma$: 
\begin{quote} 
``\emph{There exists a sequence $\vec{x}=\langle x_i \mid i<\gamma\rangle$ such that $\mathbf{1}_{\Col(\omega,\gamma)}$ forces $\vec{x}$ to be cofinal in the rank on the complement of $A$.}'' 
\end{quote} 
This is a $\Sigma_2$-formula, since ``\emph{$\vec{x}$ is cofinal in the rank on $A$}'' is $\Pi^1_2$ and thus $\Pi_1$, and the statement that a $\Pi_1$-formula is forced is $\Pi_1$. 
Since $A$ is countable, $A\subseteq L_\delta$ for some countable $\delta$ by the previous claim. 
Thus: 
\begin{quote} 
``\emph{There exist countable $\gamma$ and $\delta$ with $\gamma\leq\delta$ with $\psi(\gamma)$ and $\mathbf{1}_{\Col(\omega,\delta)}$ forces $A\subseteq L_\delta$.}''\footnote{Since $A$ does not gain any new elements in $\Col(\omega,\delta)$-generic extensions of $V$ by the proof of the previous claim, the statement ``\emph{$\mathbf{1}_{\Col(\omega,\delta)}$ forces $A\subseteq L_\delta$}'' is in fact equivalent to ``\emph{$A\subseteq L_\delta$}''. However, we need this formulation to see that the formula is $\Sigma_2$.}
\end{quote} 
In any model where $\gamma$ is countable, $A$ is $\Sigma^1_2$ in a real by Lemma \ref{countable Pi12 in inner models}, so ``\emph{$A\subseteq L_\delta$}'' is $\Pi^1_2$ in a real in $\Col(\omega,\delta)$-generic extensions. 
Therefore, ``\emph{$\mathbf{1}_{\Col(\omega,\delta)}$ forces $A\subseteq L_\delta$}'' is $\Pi_1$ and the whole formula is $\Sigma_2$. 
By the definition of $\tau$, there exist ordinals $\gamma\leq\delta <\tau$ such that $\mathbf{1}_{\Col(\omega,\delta)}$ forces $A\subseteq L_\delta$ over $L$. 
In particular, $A\subseteq L_\delta$. 
\end{proof} 

Note that $A$ is $\Pi^1_2$ and $L_{\sigma_\delta}\prec_1 L_{\omega_1}$. 
Since $A\subseteq L_\delta$ by the previous claim, $A$ is definable over $L_{\sigma_\delta}$. 
Since $\sigma_\delta<\tau$ by Lemma \ref{Sigma1 elementarity of Ltau}, we have $A\in L_{\sigma_\delta+1}\subseteq L_\tau$. 
This completes the proof for countable Cohen stable $\Sigma^1_2$-ranks. 
The remaining claim holds since one can then pick the relevant filters in $V$. 
\end{proof} 






\begin{theorem} 
\label{characterisation countable ranks} 
Suppose that $A$ is a countable $\Pi^1_2$ set. 
Then each of the following properties implies the next one: 
\begin{enumerate-(a)} 
\item 
\label{characterisation countable ranks 0} 
There exists a countable 
Cohen stable\footnote{I.e. stable for Cohen forcing in the sense of Definition \ref{def stable ranks}.}  $\Sigma^1_2$-rank on the complement of $A$. 
\item 
\label{characterisation countable ranks 1} 
$A\in L_\tau$. 
\item 
\label{characterisation countable ranks 2} 
$A$ is a subset of some countable $\Sigma^1_2$-set. 
\item 
\label{characterisation countable ranks 3} 
$A$ is a subset of some countable $\Delta^1_2$-set. 
\setcounter{enumi}{4}
\item 
\label{characterisation countable ranks 4} 
There exists a countable $\Sigma^1_2$-rank on the complement of $A$. 
\end{enumerate-(a)} 
In fact, \ref{characterisation countable ranks 1}, \ref{characterisation countable ranks 2} and \ref{characterisation countable ranks 3} are equivalent. 
Furthermore, if 
$\Sigma^1_3$ Cohen absoluteness holds,\footnote{I.e. $V\prec_{\Sigma^1_3} V[G]$ for any Cohen generic extension $V[G]$ of $V$.} then every $\Sigma^1_2$-rank is Cohen stable, so \ref{characterisation countable ranks 0}-\ref{characterisation countable ranks 4} are equivalent. 
If $\omega_1$ is inaccessible in $L$, then the implication from \ref{characterisation countable ranks 4} to \ref{characterisation countable ranks 1}-\ref{characterisation countable ranks 3} holds. 
\end{theorem} 
\begin{proof} 
\ref{characterisation countable ranks 0} $\Rightarrow$ \ref{characterisation countable ranks 1}: 
This is Lemma \ref{stable countable rank implies constructible}. 

\ref{characterisation countable ranks 1} $\Rightarrow$ \ref{characterisation countable ranks 2}: 
Recall that $\tau$ equals the (strict) supremum of $\Pi_1$-definable ordinals over $L_{\omega_1}$ by \cite[Lemma 4.1]{decisiontimes}. 
By assumption, there is a $\Pi_1$-definable ordinal $\gamma$ with $A\in L_\gamma$. 
Suppose that $\gamma$ is defined by the $\Pi_1$-formula $\theta(x)$. 
We define a $\Sigma^1_2$ set $B$ as follows. 
Suppose that $\varphi$ is a $\Pi_1$ definition of $A$. 
Let $x\in B$ if there exists countable ordinals $\gamma'<\delta$ such that the following conditions hold: 
\begin{itemize} 
\item 
$x\in L_{\gamma'}$, 
\item 
$L_\delta\models \varphi(x)$ and 
\item 
$L_\delta$ believes that $\gamma'$ is unique with $\theta(\gamma')$. 
\end{itemize} 
This is a $\Sigma^1_2$ definition. 
Since $\varphi$ and $\theta$ are $\Pi_1$, there is a countable $\delta_0$ with the following properties for all $\delta\geq\delta_0$: 
(a) if $L_\delta$ believes that $\gamma'$ is unique with $\theta(\gamma')$, then $\gamma'=\gamma$, 
and 
(b) for all $x\in L_\gamma$, $L_\delta\models\varphi(x)$ if and only if $V\models\varphi(x)$. 
To see that such a $\delta_0$ exists, take for each $x \in L_\gamma$ some $\delta_x$ such that the statements $\theta(x)$ and $\varphi(x)$ are absolute between $L_{\delta_x}$ and $V$. 
Let $\delta_0:= \sup_{x\in L_\gamma} \delta_x$.
We see $A\subseteq B$ by letting $\gamma'=\gamma$ and $\delta$ sufficiently large in the definition of $B$. 
It is clear that $B\setminus A\subseteq L_{\delta_0}$ and thus $B$ is countable.

\ref{characterisation countable ranks 2} $\Rightarrow$ \ref{characterisation countable ranks 1}: 
Suppose that $A$ is subset of a countable $\Sigma^1_2$ set $B$ defined by the $\Sigma^1_2$ formula $\psi(x)$. 
Any countable $\Sigma^1_2$ set is a subset of $L$ by the Mansfield-Solovay theorem \cite[Theorem 25.23]{Je03}. 
The $\Sigma_2$-statement $\exists \alpha\ \forall y\ (\psi(y)\Rightarrow y\in L_\alpha)$ holds in $V$. 
By the definition of $\tau$, there is some $\alpha<\tau$ with $\forall y\ (\psi(y)\Rightarrow y\in L_\alpha)$. 
Hence $A$ and $B$ are subsets of $L_\alpha$. 
Since $L_{\sigma_\alpha}\prec_1 L$, 
$A$ is definable over $L_{\sigma_\alpha}$. 
Since $\sigma_\alpha<\tau$, $A\in L_\tau$. 

\ref{characterisation countable ranks 2} $\Rightarrow$ \ref{characterisation countable ranks 3}: 
$A$ is a $\Pi^1_2$ set and there is a countable $\Sigma^1_2$ set $B$ with $A\subseteq B$. 
By the $\Sigma^1_2$ reduction property \cite[4B.10]{Mo09}, there is a $\Delta^1_2$ set $B'$ with $A\subseteq B'\subseteq B$. 

\ref{characterisation countable ranks 3} $\Rightarrow$ \ref{characterisation countable ranks 2}: 
This is clear. 

\ref{characterisation countable ranks 3} $\Rightarrow$ \ref{characterisation countable ranks 4}: 
Let $B$ be a countable $\Delta^1_2$ with $A\subseteq B$. 
Since $B\setminus A$ is $\Sigma^1_2$, it admits a $\Sigma^1_2$ rank. 
Note that this rank is countable, since $B\setminus A$ is countable. 
We adjoin the complement of $B$ as a single equivalence class at the bottom. 
This yields a countable $\Sigma^1_2$ rank on the complement of $A$. 

\ref{characterisation countable ranks 4} $\Rightarrow$ \ref{characterisation countable ranks 0}: 
Suppose that $\Sigma^1_3$ Cohen absoluteness holds. 
Take any countable $\Sigma^1_2$ rank on the complement of $A$. 
The statement that it is a $\Sigma^1_2$ rank on the complement of $A$ is $\Pi^1_3$ and it is thus stable to Cohen extensions. 

\ref{characterisation countable ranks 4} $\Rightarrow$ \ref{characterisation countable ranks 1}: 
If $\omega_1$ is inaccessible in $L$, then the claim holds by the final claim of Lemma \ref{stable countable rank implies constructible}. 
\end{proof} 

We do not know if the equivalence of \ref{characterisation countable ranks 0}-\ref{characterisation countable ranks 4} in the previous result is provable in $\ZFC$ alone. 
Note that \ref{characterisation countable ranks 1} is equivalent to $A\in L$ if $\omega_1^L$ is countable. 
However, this is not true in general: 

\begin{proposition} 
\label{ctble Pi12 cofinal in tau} 
If $\omega_1^L=\omega_1$, then there exists a countable $\Pi^1_2$ set that is an element of $L$, but not of $L_\tau$. 
\end{proposition} 
\begin{proof} 
The set collects $L$-least codes for least elements $\alpha$ of $\Pi_1$-definable subsets of $\WO$ together with the $L$-least real witnessing that $\alpha$ is least. 
In detail, let $\varphi(n,x)$ be a universal $\Pi_1$-formula, i.e., $\varphi(n,x)$ is $\Pi_1$ and $\langle \varphi_n(x) \mid n\in\omega\rangle$ enumerates all $\Pi_1$-formulas with one free variable, where $\varphi_n(x)=\varphi(n,x)$. 
Let $A$ be the $\Pi^1_1$ set of $L$-least codes for $L$-levels $L_\alpha$. 
Let $B$ be the set of $(x,y,n)\in A^2\times \omega$ such that $\alpha_x<\alpha_y$, $\varphi_n(\alpha_x)$ holds, $L_{\alpha_y}$ sees that $\varphi_n(\beta)$ fails for all $\beta<\alpha_x$, but no $\gamma<\alpha_y$ sees this. 
We first show that $B$ is a subset of $L_\tau$. 
If $\varphi_n(x)$ defines a nonempty subset $C$ of $\WO$, then $\alpha_x\leq \min(C)$ for all $x$, $y$ with $(x,y,n)\in B$. 
Therefore $\alpha_x<\tau$ and $\alpha_y<\sigma_{\alpha_x}\leq \tau$ by Lemma \ref{Sigma1 elementarity of Ltau}. 
Since $L_{\tau}\prec_1 L_{\omega_1}$ by the same lemma, we have $x,y\in L_\tau$. 
It remains to show that $B$ is unbounded in $L_\tau$.
This holds since the set of $\Pi_1$-definable singletons is unbounded in $\tau$. 
\end{proof} 

Is there an analogue of Theorem \ref{characterisation countable ranks} for Borel sets instead of countable sets? 
The first step is to ask whether every $\Delta^1_2$ Borel set has a \borel code in $L$. 
We shall discuss this in Section \ref{section Sigma12 Borel sets}.

\subsection{Countable and co-countable sets} 
\label{subsection countable sets} 

The $\Pi^1_1$ sets constructed for the lower bound in Section \ref{section - the lower bound} are all Borel. 
We now study the lengths of ranks on 
countable and co-countable sets. 
Note that finite $\Pi^1_1$ sets admit only finite ranks, while any cofinite $\Pi^1_1$ set is $\Delta^1_1$ and thus admits $\Pi^1_1$ ranks of any positive length less than $\omega_1^\ck$.

\begin{proposition} 
The maximal length of $\Pi^1_1$-ranks on co-countable $\Pi^1_1$ sets is $\omega_1^\ck$. 
Moreover, there exist co-countable $\Pi^1_1$ sets that do not admit $\Pi^1_1$-ranks shorter than $\omega_1^\ck$.
\end{proposition} 
\begin{proof} 
For the upper bound, note that any countable $\Sigma^1_1$ set $A$ is a subset of $L_{\omega_1^\ck}$ by a standard argument.\footnote{\cite[Theorem 4.3]{Hjorth-Vienna-notes-on-descriptive-set-theory}} 
Moreover, the reals of $L_{\omega_1^\ck}$ form a $\Pi^1_1$ set. 
Suppose that a $\Pi^1_1$-rank on $2^\omega\setminus A$ is given. 
By the $\Pi^1_1$ reduction property, there exists a $\Delta^1_1$ subset $B$ of $L_{\omega_1^\ck}$ that contains $A$. 
By the boundedness lemma \cite[4C.11]{Mo09},\footnote{Or by the Kunen-Martin theorem.}  
the rank's restriction to $2^\omega\setminus B$ has order type ${<}\omega_1^\ck$. 

First suppose that $2^\omega\setminus B$ contains some real that is above all elements of $B\setminus A$ in the rank. 
Then $B\setminus A$ is $\Sigma^1_1$. 
So the rank's restriction to $B\setminus A$ also has length less than $\omega_1^\ck$ and thus the overall length is less than $\omega_1^\ck$. 

Now suppose there is no such real. 
The restriction to $B\setminus A$ is a $\Pi^1_1$-rank on a countable $\Pi^1_1$ set. 
Since all its elements are $\Delta^1_1$, all initial segments are $\Delta^1_1$ and thus have length ${<}\omega_1^\ck$. 
Thus the length of the restriction to $B\setminus A$ is at most $\omega_1^\ck$. 
By the case assumption, the rank's length is at most $\omega_1^\ck$. 

For the lower bound, identify $\omega$ with a subset of $2^\omega$. 
We claim that a $\Pi^1_1$-rank on a $\Pi^1_1$ set $A$ of natural numbers that is not $\Sigma^1_1$ cannot have length below $\omega_1^\ck$. 
This suffices, since such a rank can be extended to a $\Pi^1_1$ rank on the union of $A$ with $2^\omega\setminus \omega$. 
Towards a contradiction, suppose its length is $\alpha<\omega_1^\ck$. 
Let $B$ denote the $\Pi^1_1$ set of natural numbers coding a Turing machine that determines a computable wellorder of $\omega$. 
Consider the rank on $B$ that compares natural numbers $k$ and $l$ via the order type of the computable wellorders coded by $k$ and $l$.  
We then have $n\in A$ if and only if there exists an isomorphism between the initial segment of $A$ up to $n$ and the initial segment of $B$ up to $\alpha$. 
However, $A$ would then be $\Sigma^1_1$. 
\end{proof} 

\begin{proposition} \ 
\label{ranks on countable sets} 
\begin{enumerate-(1)} 
\item 
\label{ranks on countable sets 1} 
If $\omega_1^L=\omega_1$, then the (strict) supremum of lengths of $\Pi^1_1$-ranks on countable $\Pi^1_1$ sets is $\tau$. 
Moreover, for unboundedly many $\alpha<\tau$ there exist countable $\Pi^1_1$ sets that do not admit $\Pi^1_1$-ranks shorter than $\alpha$. 

\item 
\label{ranks on countable sets 2} 
If $\omega_1^L$ is countable, then the maximal length of $\Pi^1_1$-ranks on countable $\Pi^1_1$ sets is $\omega_1^L$. 
Moreover, there exist countable $\Pi^1_1$ sets that do not admit $\Pi^1_1$-ranks shorter than $\omega_1^L$. 
\end{enumerate-(1)} 
\end{proposition} 
\begin{proof} 
\ref{ranks on countable sets 1} 
$\tau$ is a strict upper bound even for the lengths of countable $\Sigma^1_2$-ranks by Proposition \ref{Kunen-Martin for Sigma12}. 

For the lower bound, it suffices to show that for each $\nu<\tau$, there exists a countable $\Pi^1_1$ set $B$ which does not support any $\Pi^1_1$-rank of length ${<}\nu$. 
We can assume $\nu$ is $\Pi_1$-definable. 
For the set $A$ defined in the beginning of Section \ref{section - the lower bound}, any $x\in A$ is a $\mu$-index for some $\mu<\nu$ by definition. 
A standard Skolem hull argument shows that there exists some $y\in L_{\alpha_x+1}$ with $\otp(x)=\otp(y)$. 
$\Pi^1_1$ suffices to express statements in $L_{\alpha+1}$, given a code for $L_\alpha$. 
Therefore, 
the set $B$ of all $L$-minimal elements of isomorphism classes of $A$ is $\Pi^1_1$ as well. 
By Lemma \ref{no short ranks}, the length of any $\Pi^1_1$ rank on $B$ is at least $\sigma_\nu\geq \nu$.   

\ref{ranks on countable sets 2} 
To show that $\omega_1^L$ is an upper bound, note that any countable $\Sigma^1_2$ set $A$ is contained in $L$ by the Mansfield-Solovay theorem. 
Take any $\Sigma^1_2$ rank on $A$. 
By Shoenfield absoluteness, the same formulas define a $\Sigma^1_2$ rank on $A$ in $L$. 
In $L$, the Kunen-Martin theorem shows that the initial segments are countable. 
Hence the rank's length is at most $\omega_1^L$. 

For the lower bound, we show that there exists a countable $\Pi^1_1$ set $A$ such that any $\Pi^1_1$-rank on $A$ has length precisely $\omega_1^L$. 
Let $S$ denote the unbounded set of $\alpha<\omega_1^L$ such that $L_{\alpha+1}\setminus L_\alpha$ contains reals. 
For any $\alpha\in S$ there is some $x\in \WO\cap (L_{\alpha+1}\setminus L_\alpha)$ with order type $\alpha$ by \cite[Theorem 1]{boolos1969degrees}. 
Therefore, the set $A$ of reals $x\in \WO$ with $\alpha_x\in S$ that are $L$-least in their isomorphism class is again $\Pi^1_1$. 
Moreover, $A$ does not admit any $\Pi^1_1$-rank of length less than $\omega_1^L$. 
Otherwise, some equivalence class $B$ would be unbounded in $\WO^L$. 
By $\Sigma^1_1$ absoluteness this also holds in $L$, contradicting the boundedness lemma \cite[4C.10]{Mo09}. 
\end{proof} 



\begin{proposition} \ 
\label{ranks on countable Sigma12 sets} 
\begin{enumerate-(1)} 
\item 
\label{ranks on countable sets 1} 
If $\omega_1^L=\omega_1$, then the (strict) supremum of lengths of $\Sigma^1_2$-ranks on countable $\Sigma^1_2$ sets is $\tau$. 
Moreover, for unboundedly many $\alpha<\tau$ there exist countable $\Sigma^1_2$ sets that do not admit $\Sigma^1_2$-ranks shorter than $\alpha$. 

\item 
\label{ranks on countable sets 2} 
If $\omega_1^L$ is countable, then the maximal length of $\Sigma^1_2$-ranks on countable $\Sigma^1_2$ sets is $\omega_1^L$. 
Moreover, there exist countable $\Sigma^1_2$ sets that do not admit $\Sigma^1_2$-ranks shorter than $\omega_1^L$. 
\end{enumerate-(1)} 
\end{proposition} 
\begin{proof} 
\ref{ranks on countable sets 1} \ 
$\tau$ is a strict upper bound for the lengths of countable $\Sigma^1_2$-ranks by Proposition \ref{Kunen-Martin for Sigma12}. 

For the lower bound, the proof is similar to that of Proposition \ref{Sigma12 long ranks}. 
It suffices to show that for any $\Pi_1$-definable $\nu$, there exists a countable $\Sigma^1_2$ set that does not support any $\Sigma^1_2$-rank of length less than $\sigma_\nu$. 
Let $A$ denote the set defined relative to $\nu$ in the beginning of Section \ref{section - the lower bound}. 
Let $B$ denote the countable $\Pi^1_1$ set of $L$-least codes for 
$L$-levels $L_\alpha$, where $\alpha$ is countable in $L$. 
Then 
$$ C:= \{ x\in B \mid \exists y\in A\ x\in L_{\alpha_y} \} = B \cap L_{\sigma_\nu}$$ 
is a countable $\Sigma^1_2$ set. 
It suffices to show that $C$ does not support a $\Sigma^1_2$-rank of length $\alpha<\sigma_\nu$. 
Suppose otherwise. 
Then $x\in B\setminus C=B\setminus L_{\sigma_\nu}$ can be expressed as: 
\begin{quote} 
``\emph{$x\in B$ and there exists a sequence $\langle x_i \mid i<\alpha\rangle$ of reals, strictly increasing in the rank, with $\forall i<\alpha\ x \not\sqsubseteq x_i$}.'' 
\end{quote} 
This statement $\psi(x)$ is $\Sigma^1_2$ in a real in $L_{\sigma_\nu}$, since $L_{\sigma_\nu}\prec_1 L_{\omega_1}$ 
and thus $\alpha$ is countable in $L_{\sigma_\nu}$. 
Then $B\setminus L_{\sigma_\nu}$ would contain a real in $L_{\sigma_\nu}$, since $\exists x\ \psi(x)$ is $\Sigma^1_2$ and $L_{\sigma_\nu}\prec_{\Sigma^1_2} L_{\omega_1}$. 

\ref{ranks on countable sets 2} 
The upper bound was already shown in the proof of Lemma \ref{ranks on countable sets} \ref{ranks on countable sets 2}. 

For the lower bound, we show that  
any $\Sigma^1_2$-rank on the $\Sigma^1_2$ set $A:= (2^\omega)^L$ has length precisely $\omega_1^L$. 
Towards a contradiction, suppose that there exists a $\Sigma^1_2$-rank of length $\alpha<\omega_1^L$ on $A$. 
Then $A$ is $\Pi^1_2$ in a real in $L$. 
The formula $\exists x\ x \notin L$ is then $\Sigma^1_2$, contradicting Shoenfield absoluteness. 
\end{proof} 

\begin{proposition} 
\label{example Pi12 sing short rank} 
The (strict) supremum of lengths of countable $\Sigma^1_2$-ranks on co-countable $\Sigma^1_2$ sets is $\tau$. 
Moreover: 
\begin{enumerate-(1)} 
\item 
\label{example Pi12 sing short rank 1}
If $\omega_1^L=\omega_1$, then for unboundedly many $\alpha<\tau$ there exist co-countable (in fact, co-singleton) $\Sigma^1_2$ sets that do not admit $\Sigma^1_2$-ranks shorter than $\alpha$. 
\item 
\label{example Pi12 sing short rank 2}
If $\omega_1^L$ is countable and $\Sigma^1_3$ Cohen absoluteness holds,\footnote{$\Sigma^1_3$ Cohen absoluteness is used only to apply Theorem \ref{characterisation countable ranks}; we do not know if this assumption is necessary.}  then 
any co-countable $\Sigma^1_2$ set that admits a countable $\Sigma^1_2$-rank admits one of length at most $\omega_1^L$. 
Moreover, there exist co-countable $\Sigma^1_2$ sets that admit $\Sigma^1_2$-ranks of length $\omega_1^L$, but no shorter ones. 
\end{enumerate-(1)} 
\end{proposition} 
\begin{proof} 
The first claim follows from Proposition \ref{Kunen-Martin for Sigma12} and Theorem \ref{lower bound for lengths of Pi11 ranks}. 

\ref{example Pi12 sing short rank 1} 
Suppose that $\gamma$ is a $\Pi_1$-definable countable ordinal. 
Note that $x_{\sigma_\gamma}$ is $\Pi_1$-definable from $\gamma$, since it is the unique real $x\in \WO$ such that for all $\alpha$ with $(\alpha,<)$ isomorphic to $(\omega,x)$: 
\begin{itemizenew} 
\item 
$\alpha=\sigma_\gamma$, i.e. 
for every $\beta\leq\gamma$ and every $\Sigma_1$-formula $\varphi(x)$ such that $\varphi(\beta)$ holds, $\varphi(\beta)$ holds already in $L_\alpha$, but there is no $L$-level before $L_\alpha$ that satisfies the same $\Sigma_1$-formulas about ordinals $\beta\leq\gamma$, and 
\item 
$x=x_\alpha$, i.e. 
for every admissible ordinal $\delta$ with $x\in L_\delta$, $x$ is the $L$-least real in $L_\delta$ with $\otp(x)=\alpha$. 
\end{itemizenew} 
Thus $x_{\sigma_\gamma}$ is a $\Pi^1_2$-singleton. 
We claim that the complement of the singleton $x_{\sigma_\gamma}$ does not admit a $\Sigma^1_2$-rank of length less than $\sigma_\gamma$. 
Towards a contradiction, suppose it admits a $\Sigma^1_2$-rank of length $\alpha<\sigma_\gamma$. 
Note that $\alpha$ is countable in $L_{\sigma_\gamma}$, since $L_{\sigma_\gamma}\prec_1 L_{\omega_1}$. 
Thus the next formula $\psi(x)$ defining $x_{\sigma_\gamma}$: 
\begin{quote} 
``\emph{There exists a strictly increasing sequence $\langle x_i \mid i<\alpha\rangle$ in the rank with $\forall i<\alpha\ x \not\sqsubseteq x_i$.}'' 
\end{quote} 
is $\Sigma^1_2$ in a real in $L_{\sigma_\gamma}$. 
We then have $x_{\sigma_\gamma} \in L_{\sigma_\gamma}$, since $\exists x\ \psi(x)$ is $\Sigma^1_2$ and $L_{\sigma_\gamma}\prec_{\Sigma^1_2} L_{\omega_1}$. 
But $x_{\sigma_\gamma} \notin L_{\sigma_\gamma}$, since $\sigma_\gamma$ is admissible. 

\ref{example Pi12 sing short rank 2} 
For the upper bound, 
suppose that $A$ is a co-countable $\Sigma^1_2$ set that admits a countable rank and $B$ is its complement. 
Using $\Sigma^1_3$ Cohen absoluteness, 
Theorem \ref{characterisation countable ranks} yields $B\in L$. 
By the $\Sigma^1_2$ reduction property for $A$ and $(2^\omega)^L$, 
there exists a $\Delta^1_2$ superset $C$ of $B$ that is contained in $(2^\omega)^L$. 
Pick any $\Sigma^1_2$-rank on $C\setminus B$ in $L$ of length at most $\omega_1^L$ and extend it to a $\Sigma^1_2$-rank  on $A$ in $V$ by adding the complement of $C$ as the least equivalence class. 

For the lower bound, we use the next claim: 

\begin{claim*} 
Assuming $V=L$, any uncountable $\Sigma^1_2$ set has a true\footnote{Recall this means $A$ is not ${\bf \Pi}^1_2$.} $\Sigma^1_2$ subset. 
\end{claim*} 
\begin{proof} 
Suppose that $A$ is an uncountable $\Sigma^1_2$ set. 
Fix $\Sigma_1$-definable enumerations $\langle x_\alpha \mid \alpha<\omega_1 \rangle$ of $A$ and  $\langle y_\alpha \mid \alpha<\omega_1 \rangle$ of all reals without repetitions. 
Let $\varphi(x,y)$ denote a universal $\Sigma_1$-formula, so any set of reals that is $\Sigma_1$-definable in a real parameter equals $B_\alpha:=\{ x \mid \varphi(x,y_\alpha)\}$ for some $\alpha<\omega_1$. 
Let 
$$ B:= \{ x \in 2^\omega \mid \exists \alpha<\omega_1\ (x=x_\alpha \wedge \varphi(x,y_\alpha)) \}. $$ 
We now show that $B$ is a true $\Sigma^1_2$ set, so $B$ is in particular nonempty. 
By the definition of $B$, we have $x_\alpha \in B \Leftrightarrow x_\alpha\in B_\alpha$ for all $\alpha<\omega_1$. 
If $B$ fails to be a true $\Sigma^1_2$ set, then $A\setminus B$ is $\Sigma_1$ in a real parameter and thus $A\setminus B= B_\alpha$ for some $\alpha<\omega_1$, but this is an immediate contradiction. 
\end{proof} 

Let $A$ denote the $\Pi^1_1$ set of $L$-least codes for $L$-levels. 
Working in $L$, pick a $\Sigma^1_2$ subset $B$ of $A$ by the previous claim. 
Since $A=A^L$, we can assume $B=B^L$ by replacing $B$ with $A\cap B$. 
Since $\omega_1^L$ is countable, $B$ is countable. 
It suffices to show that $B$ does not admit a $\Sigma^1_2$-rank of length less than $\omega_1^L$. 
Otherwise, the rank's restriction to $L$ is a countable $\Sigma^1_2$-rank on $B$ in $L$. 
But then $B$ would be ${\bf \Pi}^1_2$ in $L$, contradicting the previous claim. 
\end{proof}

\section{$\Sigma^1_2$ Borel sets} 
\label{section Sigma12 Borel sets} 

It is not hard to come up with examples of 
$\Sigma^1_2$ Borel sets that do not have \borel codes in $L$. 
For instance, the complement of the singleton $0^\#$ does not have such a code.  
To see this, suppose that there exists a constructible \borelalpha-code for the singleton $0^\#$, where $\alpha$ is countable in $V$. 
If $\alpha$ is countable in $L$, we would have $0^\#\in L$ by Shoenfield absoluteness. 
If $\alpha$ is uncountable in $L$, pick two mutually $\Col(\omega,\alpha)$-generic filters $g$ and $h$ in $V$ over $L$. 
We then obtain a contradiction, since $0^\#$ would have to be in $L[g]\cap L[h] =L$ by Shoenfield absoluteness.\footnote{This argument works for any non-constructible $\Pi^1_2$ singleton, assuming that for every $\gamma$ that is countable in $V$ but uncountable in $L$, there exists a $\Col(\omega,\gamma)$-generic filter over $L$ in $V$, as in the argument in Lemma \ref{stable countable rank implies constructible}.} 

Some other instances of this observation work in $\ZFC$.  
For instance, a non-constructible $\Sigma^1_2$ or $\Pi^1_2$ set cannot have a ${\bf\Sigma}^{({<}\omega_1)}_2$-code in $L$, since every closed set appearing in the code would then be contained in $L$. 

The main problem regarding these observations is whether the same can happen for $\Delta^1_2$ sets. 
More precisely, we ask whether all absolutely $\Delta^1_2$ Borel sets have \borel codes in $L$. 
We phrased the problem for absolutely $\Delta^1_2$ sets in order to make a positive answer more plausible, although we do not have a counterexample for $\Delta^1_2$ sets. 
Recall that an absolutely $\Delta^1_2$ set is one whose $\Sigma^1_2$ and $\Pi^1_2$ definitions agree in all generic extensions. 
Several previous results in this direction that we present below suggest a positive answer. 

Note that a $\Delta^1_2$ Borel set with a \borel code in $L$ already has a \borel code in $L_\tau$, since the existence of such a code is a $\Sigma_2$-statement; this connects the previous problem with Theorem \ref{characterisation countable ranks}.  


\begin{remark} 
In a stronger quantitative version,   
we ask whether all absolutely $\Delta^1_2$ sets that are ${\bf \Sigma}^0_\alpha$ have 
${\bf \Sigma}^{({<}\omega_1)}_\alpha$ codes in $L$. 
A positive solution to this version 
is analogous  to Louveau's celebrated separation theorem for $\Sigma^1_1$ sets \cite[Theorem A]{louveau1980separation}.
Such a result would unify and strengthen a number of classical and recent results in descriptive set theory. 
We list a few results that can be immediately obtained as corollaries as in the above observations. 
\begin{itemize} 
\item 
{\bf Shoenfield} absoluteness;\footnote{We mean $\Sigma^1_2$ absoluteness without parameters; the full version follows from a positive solution relative to reals.} 
this is equivalent to the statement that $L$ contains all $\Pi^1_1$ singletons, 
since any nonempty $\Pi^1_1$ set contains a $\Pi^1_1$ singleton by $\Pi^1_1$ uniformisation. 
\item 
{\bf Mansfield's} and {\bf Solovay's} theorem (see \cite[Theorem 25.23]{Je03}) stating that every countable $\Sigma^1_2$ set is contained in $L$; 
equivalently, every countable $\Pi^1_1$ set is contained in $L$ by $\Pi^1_1$-uniformisation. 
\item 
{\bf Kechris'}, {\bf Marker's} and {\bf Sami's} result that any $\Pi^1_1$ Borel set has an \borel code in $L_\tau$ \cite{MR1011178} (see Proposition \ref{sup of Borel ranks of Pi11 sets} below). 
\item 
{\bf Stern's} result that assuming $\omega_1$ is inaccessible in $L$, every (absolutely)\footnote{A positive answer for all $\Delta^1_2$ Borel sets would yield this claim for all $\Delta^1_2$ sets as in Stern's result.} $\Delta^1_2$ Borel set has a \borel code in $L$ (see \cite[Section 0.2, Theorem 2]{stern1984lusin} and Theorem \ref{Separating Sigma12-sets by a generic Borel set} below).  
\item 
{\bf Kanovei's} and {\bf Lyubetsky's} theorem that Cohen, random and Sacks forcing do not add new $\Delta^1_2$ Borel sets \cite[Theorem 1]{kanovei2019borel}. 
\end{itemize} 
\end{remark}





\subsection{$\Pi^1_1$ sets}

We first note where countable $\Pi^1_1$ and $\Sigma^1_2$ sets appear in $L$. 

\begin{proposition} 
The least ordinal $\alpha$ such that every countable $\Pi^1_1$ set $A$ is an element of $L_\alpha$ equals $\tau$ 
if $\omega_1^L$ is uncountable and $\omega_1^L+1$ otherwise. 
The same holds for countable $\Sigma^1_2$ sets. 
\end{proposition} 
\begin{proof} 
First suppose that $\omega_1^L$ is uncountable. 
For the upper bound,\footnote{This part does not use that $\omega_1^L$ is uncountable.}  note that every countable $\Sigma^1_2$ set is a subset of $L$ by the Mansfield-Solovay theorem. 
By Shoenfield absoluteness, $A$ is definable over $L$ and therefore, $A\in L$. 
Since $A$ is countable, there exists a countable ordinal $\alpha$ with $A\subseteq L_\alpha$. 
Since the statement ``\emph{$A\subseteq L_\alpha$}'' is $\Pi^1_2$, 
there exists some $\alpha<\tau$ with $A\subseteq L_\alpha$ by the definition of $\tau$. 
$\Pi^1_1$ sets witnessing the lower bound were constructed in the proof of Proposition \ref{ranks on countable sets} \ref{ranks on countable sets 1}. 

Now suppose that $\omega_1^L$ is countable. 
For the upper bound, 
note that any countable $\Sigma^1_2$ set is a subset of $L_{\omega_1^L}$ by the Mansfield-Solovay theorem. 
By Shoenfield absoluteness, it is definable over $L_{\omega_1^L}$ and thus an element of  $L_{\omega_1^L+1}$. 
For the lower bound, note that the set of $L$-least codes for $L$-levels is a countable $\Pi^1_1$ set whose elements have unbounded $L$-ranks below $\omega_1^L$. 
\end{proof} 

It is easy to construct a $\Pi^1_1$ set in $L$ that is not Borel in $L$, but it is countable in a generic extension of $L$;\footnote{This is discussed further in \cite[Section 6.3]{sami1984sigma11}.}  
for example, the usual counterexample to the perfect set property for $\Pi^1_1$ sets in $L$ has this property, since it becomes countable after collapsing $\omega_1$ over $L$. 
Hence it is consistent that there exist $\Pi^1_1$ Borel sets with no \borelomega-codes in $L$. 
The next proposition shows that this is impossible for \borel-codes. 

\begin{proposition}\footnote{This result is implicit in the proofs of \cite[Lemma 1.2 \& Theorem 1.4]{MR1011178}.} 
\label{sup of Borel ranks of Pi11 sets} 
Any $\Pi^1_1$ Borel set has a \borel code in $L_\tau$. 
Moreover, the Borel ranks of $\Pi^1_1$ Borel sets are unbounded in $\tau$. 
\end{proposition} 
\begin{proof} 
Suppose that $A$ is a $\Pi^1_1$ Borel set. 
Then there exists a computable function $f\colon 2^\omega \rightarrow 2^\omega$ with $A=f^{-1}(\WO)$. 
Since $A$ is Borel, $f[A]$ is contained in $\WO_\alpha$ for some $\alpha<\omega_1$ by the boundedness lemma. 
Since $A$ is $\Pi^1_1$, the statement ``$\alpha$ is an upper bound for $f[A]$'' is $\Pi_1$. 
Hence there exists an upper bound $\alpha<\tau$. 
Since $\WO_\alpha$ has a Borel${}^{({<}\tau)}$-code in $L_\tau$ by induction for all $\alpha<\tau$, $A$ does. 

To show that the Borel ranks are unbounded in $\tau$, 
suppose that $\delta>\omega^\alpha$ is a $\Pi_1$-singleton defined by $\varphi(x)$. 
Let 
$$A=\{(x,y)\in \WO^2\mid \text{$\alpha_y$ is the least $\alpha$ such that } L_\alpha\models ``\varphi \text{ defines } \alpha_x" \} .$$
Let further $\xi>\delta$ be least with $L_\xi\models ``\varphi$ defines $\delta"$. 
Note that for any $(x,y)\in A$, we have $\alpha_x\leq\delta$ and $\alpha_y\leq\xi$. 
Since for each $\xi$, the set $\WO_\xi$ of \borelomega codes for $\xi$ is Borel, 
$A$ is a countable union of Borel sets and thus Borel. 
Letting $y=x_\xi$, we obtain the slice $\WO_\delta$. 
By \cite{sternborelrank},\footnote{See \cite[Lemma 1.3]{MR1011178}} $\WO_\delta$ has Borel rank at least $\alpha$.  
\end{proof} 

However, the proof does not provide a quantitative result. 
Does every $\Pi^1_1$ set that is ${\bf \Sigma}^0_\alpha$ have a 
${\bf \Sigma}^{({<}\omega_1)}_\alpha$ code in $L$? 
We shall see in Lemma \ref{weak sigma02} below that this is true for ${\bf \Sigma}^0_2$ sets, assuming $\Sigma^1_3$ Cohen absoluteness.

\subsection{$\Delta^1_2$ sets} 

The above problem has a positive answer both if $V$ is sufficiently far away from $L$ and if $V$ is sufficiently close to $L$. 
Stern gave a positive answer assuming $\omega_1$ is inaccessible in $L$ (see Theorem \ref{Separating Sigma12-sets by a generic Borel set} below). 
We now present a concise version of his argument for the reader. 

\begin{definition} 
Suppose that $\PP$ is a forcing, $p\in \PP$ and $\sigma$ is a $\PP$-name for a ${\bf \Sigma}^{(\omega)}_\alpha$ code. 
Then 
$$  B_{(\sigma/p)} := \{ x\in 2^\omega \mid \exists q\leq p\ \ q\Vdash_\PP x\in B_\sigma \}$$ 
is called a \emph{$\PP$-approximate ${\bf \Sigma}^0_\alpha$ set}. 
\end{definition} 


The next two results are implicit in \cite[Section 5]{stern1984lusin}. 

\begin{lemma}[Stern] 
\label{Stern lemma} 
Suppose that $M$ is a transitive model of $\KP$, $\PP$ is a forcing with $\tc(\PP)\in M$ and $\lambda=|\PP|^M\geq\omega$. 
\label{generic Sigma_alpha in L} 
Consider the following classes of sets: 
\begin{enumerate-(a)} 
\item 
\label{generic Sigma_alpha in L 1} 
$\PP$-approximate ${\bf \Sigma}^0_\alpha$ sets relative to names in $M$. 
\item 
\label{generic Sigma_alpha in L 2}
${\bf \Sigma}_\alpha^{(\lambda)}$ sets with codes in $M$. 
\end{enumerate-(a)} 
Then \ref{generic Sigma_alpha in L 1} is included in \ref{generic Sigma_alpha in L 2}. 
Moreover, the classes are equal if $\PP$ forces $\lambda$ to be countable. 
\end{lemma} 

\begin{proof} 
\ref{generic Sigma_alpha in L 1}$\curvearrowright$\ref{generic Sigma_alpha in L 2}: 
We construct a translation by recursion in $M$ that sends a pair $(\sigma,p)$ as above to a ${\bf \Sigma}_\alpha^{(\lambda)}$-code for $B_{(\sigma/p)}$. 

For $\alpha=1$, suppose that $\sigma\in M$ is a $\PP$-approximate ${\bf \Sigma}_1^{(\omega)}$-code below $p$ for a ${\bf\Sigma}^0_1$ set of the form $\bigcup_{n\in\omega} N_{t_n}$, where $\vec{t}=\langle t_n \mid n\in\omega\rangle$ is a sequence in $2^{<\omega}$ 
and $\langle \dot{t}_n \mid n\in\omega \rangle\in M$ is a sequence of names for the $t_n$. 
Let $T$ denote the set of all $t\in 2^{<\omega}$ such that $q\Vdash \dot{t}_n=t$ for some $q\leq p$ and some $n\in\omega$. 
Since the atomic forcing relation is absolute to $M$, we have $T\in M$. 
As $B_{(\sigma/p)} = \bigcup_{t\in T} N_t$, 
we thus obtain a ${\bf \Sigma}^{(\omega)}_1$-code in $M$ for $B_{(\sigma/p)}$. 

For $\alpha>1$, suppose that $\vec{\alpha}=\langle \alpha_n \mid n\in\omega\rangle\in M$ is a sequence of ordinals below $\alpha$ and $\vec{\sigma}=\langle \sigma_n \mid n\in\omega\rangle\in M$ is a sequence of names, where $\sigma_n$ is a name for an ${\bf \Sigma}^{(\omega)}_{\alpha_n}$-code. 
Suppose that $\sigma$ is a name for a \borelomega code for the union of the complements of the sets $B_{\sigma_n}$. 
Recall that $x\in B_{(\sigma/p)}$ if and only if some $q\leq p$ forces $x\in B_\sigma$. 
Thus 
\begin{align*}
x\in B_{(\sigma/p)} &  \Longleftrightarrow \exists q\leq p\ \  q \Vdash_\PP x\in B_\sigma  \\ 
&  \Longleftrightarrow  \exists r\leq p\ \ \exists n\in\omega \ \   r\Vdash x \notin B_{\sigma_n}  \\ 
&  \Longleftrightarrow  \exists r\leq p\ \ \exists n\in\omega \ \ \neg [   \exists s\leq r\ \ s \Vdash x \in B_{\sigma_n} ] \\ 
&  \Longleftrightarrow  \exists r\leq p\ \ \exists n\in\omega \ \ x \notin B_{(\sigma_n/r)} \\ 
\end{align*} 
The last equivalence holds by the definition of $B_{(\sigma_n/r)}$. 
Since $|\PP|^M=\lambda$, these equivalences provide a ${\bf \Sigma}^{(\lambda)}_\alpha$-code in $M$ for $B_{(\sigma/p)}$. 

\ref{generic Sigma_alpha in L 2}$\curvearrowright$\ref{generic Sigma_alpha in L 1}: 
If $\PP$ forces $\lambda$ to be countable, then one can use a name for a a bijection between $\omega$ and $\lambda$ to define a name for a ${\bf \Sigma}^{(\omega)}_\alpha$-code from a ${\bf \Sigma}_\alpha^{(\lambda)}$-code. 
\end{proof} 

Write $\alpha^*:=\alpha-1$ if $\alpha$ is finite and $\alpha\geq1$, and $\alpha^*:=\alpha$ if $\alpha$ is infinite. 

\begin{theorem}[Stern] 
\label{Separating Sigma12-sets by a generic Borel set} 
If disjoint $\Sigma^1_2$ sets $A$, $B$ can be separated by a ${\bf \Sigma}^0_\alpha$ set for $\alpha\geq1$, then they can be separated by a  ${\bf \Sigma}_\alpha^{(\omega_{\alpha^*}^L)}$ set with a code in $L$. 
%
(This relativises to reals.) 
In particular, if $\omega_1$ is inaccessible in $L$, then every $\Delta^1_2$ Borel set has a \borel-code in $L$. 
\end{theorem} 
\begin{proof} 
Suppose that $A$ and $B$ are disjoint $\Sigma^1_2$ sets and $C$ is a ${\bf \Sigma}^0_\alpha$ set that separates them. 
We shall also write $A$, $B$ and $C$ for the sets with the same definitions in other models. 
Fix a $\Col(\omega,\omega_1)$-generic filter $G$ over $V$. 
In $V[G]$, $C$ separates $A$ from $B$ by Shoenfield absoluteness. 

Since $G$ is generic over $L$, there exists a real $x_0\in L[G]$ that codes $\omega_1^V$. 
Recall that any $\Sigma^1_2$ set can be understood as given by a $\Sigma_1$ definition over $L_{\omega_1}[x]$ in the input $x$. 
Applied to $A$ and $B$ in $V[G]$, we obtain subsets $A'$, $B'$ of $A$ and $B$, respectively, by restricting the definitions to $L_{\omega_1^V}[x]$. 
Then 
$A'$ and $B'$ are $\Sigma^1_1$-definable in $x_0$ and 
$$A'\cap V=A\cap V$$ 
$$B'\cap V= B\cap V.$$ 
Recall that in $V[G]$, there exists a ${\bf \Sigma}^0_\alpha$ set, namely $C$, separating $A'$ from $B'$. 
By Shoenfield absoluteness for $L[G]$ and $V[G]$, there exists some ${\bf \Sigma}^{(\omega)}_\alpha$-code for a set with this property in $L[G]$. 
Let $\sigma\in L$ be a nice $\PP$-name for a ${\bf \Sigma}^{(\omega)}_\alpha$-code with 
$$ \text{$\Vdash_{\Col(\omega,\omega_1)}^L $``\emph{$B_\sigma$ separates $A'$ from $B'$}\ ''.} $$ 
Since the forced statement is $\Pi^1_1$ in $x_0$, it is also forced over $V$ by $\Sigma^1_1$-absoluteness between $\Col(\omega,\omega_1)$-generic extensions $L[G]$ and $V[G]$. 
Therefore in $V$, $B_{(\sigma,1)}$ 
separates $A$ from $B$. 

By Lemma \ref{generic Sigma_alpha in L}, $B_{(\sigma,1)}$ has a ${\bf \Sigma}_\alpha^{(\omega_1)}$-code in $L$. 
It remains to show that $B_{(\sigma,1)}$ has a ${\bf \Sigma}_\alpha^{(\omega_{\alpha^*}^L)}$-code in $L$. 
In fact, take any ${\bf \Sigma}^{(\infty)}_\alpha$-code in $L$. 
We can assume that each \borelinfty-code appears at most once within the code for a single union or intersection by recursively removing additional ones. 
Every ${\bf\Sigma}_1^{(\infty)}$ set is open and thus has a ${\bf\Sigma}_1^{(\omega)}$-code. 
It follows by induction that every ${\bf \Sigma}^{(\infty)}_\alpha$-code in $L$ is in fact a ${\bf \Sigma}^{(\omega_{\alpha^*}^L)}_\alpha$-code. 
\end{proof} 

We now show that proper forcing does not introduce new absolutely $\Delta^1_2$ Borel sets. 
Hence the above problem has a positive solution if $V$ is an extension of $L$ by proper forcing. 
This is inspired by a result by Kanovei and Lyubetsky \cite[Theorem 1]{kanovei2019borel} for Cohen, random and Sacks forcing. 


\begin{theorem} 
\label{proper forcing no new sets} 
Suppose that $V$ is a generic extension of 
$M$ by proper forcing. 
Then any absolutely $\Delta^1_2$ set that is ${\bf\Sigma}^0_\alpha$ has a ${\bf\Sigma}^{(\omega)}_\alpha$-code in $M$.\footnote{If $V$ is a $\PP$-generic extension of $M$, then we only need that the $\Delta^1_2$-definition remains valid in all $\PP$-generic extensions of $V$.}  
\end{theorem} 
\begin{proof} 
Suppose that $V$ is an extension of $M$ by a proper forcing $\PP\in M$. 
For any formula $\varphi(x)$, write 
$S_\varphi=\{x\in2^\omega\mid \varphi(x)\}$. 
Suppose that $p\in \PP$ forces over $M$ that $S_\varphi$ is a ${\bf \Sigma}^0_\alpha$ set. 
Let $\PP^2=\PP_0\times\PP_1$, where $\PP_i=\PP$ for $i\leq 1$. 
Let $\sigma_0$, $\sigma_1$ be $\PP^2$-names for the evaluations of $\sigma$ by the two generics, 
i.e., $\sigma_i^{G_0\times G_1}=\sigma^{G_i}$ for any $\PP^2$-generic filter $G_0\times G_1$ over $V$. 

Consider the equivalence relation $\sim$ on the set of \borelomega codes defined by $x\sim y$ if $B_x=B_y$. 
It is absolute, as it is $\Pi^1_1$. 

\begin{claim*} 
$(1,1)\Vdash_{\PP^2} \sigma_0\sim \sigma_1$. 
\end{claim*} 
\begin{proof} 
Towards a contradiction, suppose there exist $p,q\in \PP$ with 
$(p,q)\Vdash_{\PP^2} \sigma_0\not\sim \sigma_1$.  
Let $G\times H$ be $\PP^2$-generic over $V$ with $p\in G$ and $q\in H$. 
Then $\sigma^G \not\sim \sigma^H$ and thus $B_{\sigma^G}\neq B_{\sigma^H}$ in $V[G\times H]$. 
But both $B_{\sigma^G}$ and $B_{\sigma^H}$ agree with the absolute $\Delta^1_2$ definition in $V[G\times H]$, since this agreement is a $\Pi^1_2$ statement. 
\end{proof} 

We now show that densely many $q \in \PP$ force that $\sigma$ is equivalent to a \borelomega code in $V$. 
Fix any $p\in \PP$. 
Since $\PP$ is proper, there exist a cardinal $\theta>2^{|\PP|}$, 
a countable substructure $M\prec H_\theta$ with $\PP, p, \sigma\in M$ and a master condition $q\leq p$ for $M$, i.e. such that for every predense set $A\in M$, $A\cap M$ is predense below $q$. 
Let $g\in V$ be $(\PP\cap M)$-generic over $M$ and $H$ $\PP$-generic over $V$ with $q\in H$. 
Since $q$ is a master condition, $(H\cap M)$ is $(\PP\cap M)$-generic over $M$. 
The next claim suffices. 


\begin{claim*} 
For all $(\PP\cap M)$-generic filters $h,k$ over $M$, we have $\sigma^h\sim \sigma^k$. 
\end{claim*} 
\begin{proof}  
Let $\bar{M}$ denote the transitive collapse of $M$, $\bar{\sigma}$ the preimage of $\sigma$, and $\bar{\PP}$, $\bar{h}$, $\bar{k}$ the pointwise preimages of $(\PP\cap M)$, $h$ and $k$.\footnote{The reason for working with transitive models is that the product lemma can fail for forcing over non-transitive models.}  
Then $\bar{\sigma}^{\bar{h}}=\sigma^h$ and $\bar{\sigma}^{\bar{k}}=\sigma^k$. 
Let $\bar{l}$ be $\bar{\PP}$-generic over both $\bar{M}[\bar{g}]$ and $\bar{M}[\bar{h}]$ and $l$ its pointwise image. 
Again, $\bar{\sigma}^{\bar{l}}=\sigma^l$. 
By the previous claim and absoluteness of $\sim$, we have $\bar{\sigma}^{\bar{h}}\sim \bar{\sigma}^{\bar{l}}\sim \bar{\sigma}^{\bar{k}}$ and hence $\sigma^h\sim \sigma^l\sim \sigma^k$. 
\end{proof} 
Thus $\sigma^H\sim \sigma^g\in V$ as required. 
\end{proof} 


The next lemma shows that the main problem formulated above 
has a positive answer in some extensions of $L$ where $\omega_1^L$ is collapsed, but $\omega_1$ is accessible in $L$. 
For instance, this can happen if $\omega_1=\omega_2^L$. 
By a \emph{collapse}, we mean a forcing of the form $\Col(\omega,\nu)$ for some cardinal $\nu$. 

\begin{proposition} 
Suppose that $V$ is an extension of $M$ by a collapse. 
Then any absolutely $\Delta^1_2$ set that is ${\bf\Sigma}^0_\alpha$ has a ${\bf\Sigma}^{({<}\omega_1)}_\alpha$-code in $M$. 
\end{proposition} 
\begin{proof} 
Suppose that $\PP$ is a collapse in $M$ and 
$V$ is a $\PP$-generic extension of $M$. 
Suppose further that $A$ is an absolutely $\Delta^1_2$ set. 
We shall also write $A$ for the set with the same definition in other models. 

Suppose that $\sigma\in M$ is a $\PP$-name for a ${\bf\Sigma}^{(\omega)}_\alpha$-code such that 
$\mathbf{1}\Vdash_\PP^M B_\sigma=A$. 
Note that the statement $B_\sigma=A$ is $\Pi^1_2$, since $A$ is absolutely $\Delta^1_2$. 
Hence it is persistent to outer models. 
We thus have $1\Vdash_\PP B_\sigma=A$ in $V$. 
Therefore for all $p,q\in \PP$, 
$$ p\Vdash x\in B_\sigma \Longleftrightarrow p\Vdash x\in A \Longleftrightarrow q\Vdash x\in A   \Longleftrightarrow  q\Vdash x\in B_\sigma. $$ 
Thus $A$ is a $\PP$-approximate ${\bf \Sigma}^0_\alpha$ set with a name in $M$. 
By Lemma \ref{generic Sigma_alpha in L}, $A$ has a ${\bf \Sigma}^{(\nu)}_\alpha$-code in $M$, where $|\PP|^M=\nu$. 
Since $\PP$ is a collapse, $\nu$ is countable in $V$. 
\end{proof}

The next result is an attempt to weaken (in consistency strength) Stern's assumption that $\omega_1$ is inaccessible in $L$. 

\begin{proposition} 
\label{weak sigma02} 
If $\Sigma^1_3$ Cohen absoluteness holds, then 
any absolutely $\Delta^1_2$ set $A$ that is ${\bf \Sigma}^0_2$ has a 
${\bf \Sigma}^{({<}\omega_1)}_2$-code in $L$. 
\end{proposition} 
\begin{proof} 
By Lemma \ref{Separating Sigma12-sets by a generic Borel set}, $A$ has a ${\bf \Sigma}^{(\omega_1^L)}_2$-code in $L$. 
We can thus assume $\omega_1^L=\omega_1$.  
$A$ equals the union of all sets $[T]$, where $T\in L$ is a subtree of $2^{<\omega}$ with $[T]\subseteq A$. 
Let $\mathcal{T}$ denote the collection of all such trees $T$. 
Suppose $A=\bigcup_{n\in\omega} [T_n]$. 
Let $\inter_X(Y)$ denote the interior of $Y$ in the space $X$. 
For each $T\in \mathcal{T}$ and $n\in\omega$, let $U_n=\bigcup_{T\in \mathcal{T}} \inter_{[T_n]}([T\cap T_n])$. 
Since the sets $\inter_{[T_n]}([T\cap T_n])$ are open in $[T_n]$, countably many suffice to cover $U_n$. 

\begin{claim*} 
$[T_n]\setminus U_n$ is countable. 
\end{claim*} 
\begin{proof} 
Otherwise, let $S$ be the unique perfect subtree of $T_n$ such that $[S]$ is the perfect kernel of $[T_n]\setminus U_n$. 
By the definition of $U_n$, $[S\cap T]$ is nowhere dense in $[S]$ for any $T\in \mathcal{T}$. 
We now force with $S$ over $L[S]$. 
There exists a generic filter in $V$ by $\Sigma^1_3$ Cohen absoluteness. 
We thus find a real $x\in [S] \setminus \bigcup_{T\in \mathcal{T}}[T]$. 
However, this contradicts the fact that $[S]\subseteq [T_n]\subseteq A \subseteq \bigcup_{T\in \mathcal{T}}[T]$. 
\end{proof} 

For each $n\in\omega$, there exists a countable subset $\mathcal{T}_n$ of $\mathcal{T}$ with $[T_n]\subseteq \bigcup_{T\in \mathcal{T}_n} [T]$ by the previous claim. 
Hence there exists a countable subset $\mathcal{T}^*$ of $\mathcal{T}$ with $A \subseteq \bigcup_{T\in \mathcal{T}^*} [T]$. 


Since $\omega_1^L=\omega_1$, then $\mathcal{T}^*$ is contained in $\mathcal{T}\cap L_\alpha$ for some countable $\alpha$. 
Then the equality 
$$ A = \bigcup_{T\in \mathcal{T}\cap L_\alpha,\ [T]\subseteq A} [T] $$  
induces a ${\bf \Sigma}^{(\alpha)}_2$-code in $L$ for $A$. 
\end{proof} 



\subsection{$\Sigma^1_2$ sets} 

While any $\Delta^1_2$ Borel set with a \borel code in $L$ already has one in $L_\tau$, this is not the case for $\Sigma^1_2$ sets. 

\begin{proposition} 
\label{Sigma12 Borel of rank tau} 
There exists a $\Sigma^1_2$ Borel set of rank $\tau$ with a Borel${}^{(\tau)}$ code in $L$. 
\end{proposition} 
\begin{proof} 
The proof is similar to Proposition \ref{ctble Pi12 cofinal in tau}. 
We define a $\Pi^1_2$ set $A$ Borel set of rank $\tau$ with a Borel${}^{(\tau)}$ code in $L$. 
It collects codes for least elements $\alpha$ of $\Pi_1$-definable subsets of $\WO$ together with a real witnessing that $\alpha$ is least. 
In detail, let $\varphi(n,x)$ be a universal $\Pi_1$-formula, i.e., $\varphi(n,x)$ is $\Pi_1$ and $\langle \varphi_n(x) \mid n\in\omega\rangle$ enumerates all $\Pi_1$-formulas with one free variable, where $\varphi_n(x)=\varphi(n,x)$. 
$A$ denotes the set of $(x,y,n)\in \WO^2\times \omega$ such that, $\alpha_x<\alpha_y$, $\varphi_n(\alpha_x)$ holds, $L_{\alpha_y}$ sees that $\varphi_n(\beta)$ fails for all $\beta<\alpha_x$, but no $\gamma<\alpha_y$ sees this. 
Note that if $\varphi_n(x)$ defines a nonempty subset $B$ of $\WO$, then $\alpha_x\leq \min(B)$ for all $x$, $y$ with $(x,y,n)\in A$. 
Therefore $\alpha_x<\tau$ and $\alpha_y<\sigma_{\alpha_x}\leq \tau$ by Lemma \ref{Sigma1 elementarity of Ltau}. 
Thus $A$ is the union of slices of the form $\WO_\alpha\times \WO_\beta$ for unboundedly many $\alpha,\beta<\tau$. 
Since $\WO_{\leq\gamma}$ is not $\Pi^0_{2\cdot\delta}$ if $\gamma=\omega^\delta$ \cite{sternborelrank},\footnote{See \cite[Lemma 1.3]{MR1011178}.} it follows that the Borel rank of $A$ equals $\tau$. 
\end{proof}

\section{Conclusion and open questions} 

In Section \ref{section lengths of ranks}, we constructed various examples where the least rank on a $\Pi^1_1$ or $\Sigma^1_2$ set has a certain minimal countable length such as $\sigma$, $\sigma_\sigma$ etc. 
Assuming $\omega_1^L=\omega_1$, one can similarly obtain countable  sets with a unique length of ranks. 
These phenomena deserve to be explored further: at precisely which countable ordinals does a new minimal length appear? 
Assuming $\omega_1^L=\omega_1$, which ordinals appear as the unique length of ranks on a set? 

Proposition \ref{ranks on countable sets} 
shows that there is no provable bound for the lengths of countable ranks on $\Pi^1_1$ Borel sets of fixed rank. 
Is this also true assuming large cardinals? 
Is there a result in the converse direction: can one compute a bound for Borel ranks of $\Pi^1_1$ sets that admit ranks of a fixed countable length? 
 

In Section \ref{section Sigma12 Borel sets}, we mentioned our main open question about $\Delta^1_2$ Borel sets: 

\begin{question} 
Does every absolutely $\Delta^1_2$ Borel set have a \borel-code in $L$? 
\end{question} 

Stern proved this assuming $\omega_1$ is inaccessible in $L$. 
The only other similar result known to us is Louveau's separation theorem, but his proof technique fails here. 
Stern's result motivates us to ask for a similar generalisation of Theorem \ref{characterisation countable ranks}: 

\begin{question} 
Does a $\Sigma^1_2$ Borel set have a \borel-code in $L_\tau$ if and only if it admits a countable $\Sigma^1_2$-rank? 
\end{question} 

For instance, one can ask if this holds at the first and second levels of the Borel hiearchy. 
It is not clear if ranks are strong enough to guarantee the existence of a \borel-code in $L_\tau$. 
However, sets that admit ranks of the following form come with a Borel definition induced by the layers. 
Call a $\Sigma^1_2$-rank on a set $A$ \emph{uniform} if the initial segments are uniformly $\Sigma^1_1$ and $\Pi^1_1$ in the sense that there are $\Sigma^1_1$ and $\Pi^1_1$ formulas $\varphi(x,y)$ and $\psi(x,y)$ such that for each $y \in \WO_\alpha$, the statement that $x$ has rank at most $\alpha$ is equivalent to both $\varphi(x,y)$ and $\psi(x,y)$ and moreover, these formulas define a subset of $A$ for any fixed $y$ with $\otp(y)$ in the rank's range. 
It is easy to see that ranks induced by infinite time computations satisfy this property. 

\begin{question} 
Does a $\Sigma^1_2$ set have a \borel-code in $L_\tau$ if and only if it admits a uniform countable $\Sigma^1_2$-rank? 
\end{question}


Various other natural questions arise regarding the results in Sections \ref{section lengths of ranks} and \ref{section Sigma12 Borel sets}. 
What are the lengths of $\Pi^1_1$- and $\Sigma^1_2$-ranks on open and closed sets? 
Do the suprema of $L$-levels of countable $\Pi^1_2$ sets in $L$ and that of Borel ranks of $\Sigma^1_2$ Borel sets with a \borel code in $L$ equal $\tau$? 
Theorem \ref{characterisation countable ranks} \ref{characterisation countable ranks 3} suggests to ask whether every constructible $\Pi^1_2$ singleton contained in a countable $\Pi^1_1$ set, or otherwise how one can characterise $\Pi^1_2$ singletons with this property. 

Furthermore, we ask for generalisations of the present results to all projective levels. 
Firstly, one can ask whether the periodic pattern in Figure \ref{figure length of pwos} repeats itself at the next level: 

\begin{question} 
Assume $0^\#$ exists. 
Are the suprema of the values in the third column of Figure \ref{figure length of pwos} equal? 
\end{question} 

Assuming the axiom of projective determinacy (\PD), one can hope for results analogous to Theorem \ref{main theorem} and those shown in Figure \ref{figure length of pwos} for all projective levels. 
At the level of $\Pi^1_3$ sets, we suggest to combine the proofs in this paper with techniques from inner model theory and replace $L$ by $M_1$, an inner model with a Woodin cardinal. 

\begin{question} 
Assume \PD. What is the supremum of lengths of countable $\Pi^1_3$-ranks? 
\end{question} 
 
Studying the lengths of projective prewellorders further suggests to study the related problem of proving basis theorems in order to find representatives for equivalence classes of prewellorders as in \cite{hjorth1993thin, schlicht2014thin}. 


\bibliographystyle{alpha}
\bibliography{References}
\end{document}